\documentclass[12pt,a4paper]{article}
\usepackage{amsmath,amssymb,amsthm,graphicx,color,bbm}
\usepackage[a4paper,left=1in,right=1in,top=1in,bottom=1in]{geometry}
\usepackage{cite}
\usepackage[british]{babel}
\usepackage{hyperref} 
\usepackage{enumitem}
\usepackage{tikz}
\usetikzlibrary{decorations.markings}

\hypersetup{
    colorlinks=false,
    pdfborder={0 0 0},
}

\newtheorem{theorem}{Theorem}[section]
\newtheorem{corollary}[theorem]{Corollary}
\newtheorem{proposition}[theorem]{Proposition}
\newtheorem{lemma}[theorem]{Lemma}

\numberwithin{equation}{section}

\theoremstyle{definition}
\newtheorem{definition}[theorem]{Definition}

\newtheorem{problem}[theorem]{Problem}

\theoremstyle{remark}
\newtheorem{remark}[theorem]{Remark}
\newtheorem{remarks}[theorem]{Remarks}
\newtheorem*{remark*}{Remark}

 \allowdisplaybreaks

\newcommand{\1}[1]{{\mathbbm{1}\mkern -1.5mu}{\{#1\}}}
\newcommand{\R}{{\mathbb R}}
\newcommand{\Z}{{\mathbb Z}}
\newcommand{\N}{{\mathbb N}}
\newcommand{\ZP}{{\mathbb Z}_+}
\newcommand{\RP}{{\mathbb R}_+}

\newcommand{\bbX}{{\mathbb X}}

\DeclareMathOperator{\Exp}{\mathbb{E}}
\DeclareMathOperator{\tExp}{{\widetilde \Exp}}
\let\Pr\relax
\DeclareMathOperator{\Pr}{{\mathbb P}}
\let\Re\relax
\DeclareMathOperator{\Re}{{\mathrm{Re}}}
\DeclareMathOperator{\bbQ}{{\mathbb Q}}
\DeclareMathOperator{\tPr}{{\widetilde \Pr}}
\DeclareMathOperator{\Var}{\mathbb{V}ar}

\newcommand{\gcr}{{\gamma_{\mathrm{c}}}}
\newcommand{\supp}{\mathop \mathrm{supp}}

\newcommand{\tra}{{\scalebox{0.6}{$\top$}}}

\newcommand{\eps}{\varepsilon}

\newcommand{\re}{{\mathrm{e}}}

\newcommand{\ud}{{\mathrm d}}

\newcommand{\cF}{{\mathcal F}}

\newcommand{\tZ}{{\widetilde Z}}
\newcommand{\tS}{{\widetilde S}}
\newcommand{\tp}{{\widetilde p}}
\newcommand{\tmu}{{\widetilde \mu}}

\newcommand{\bigmid}{\; \bigl| \;}

\newcommand{\taue}{{\tau_\mathcal{E}}}

\makeatletter
\def\namedlabel#1#2{\begingroup  
    (#2)%
    \def\@currentlabel{#2}%
    \phantomsection\label{#1}\endgroup
}
\makeatother

\newlist{myenumi}{enumerate}{10}
\setlist[myenumi]{leftmargin=0pt, labelindent=\parindent, listparindent=\parindent, labelwidth=0pt, itemindent=!, itemsep=1pt, parsep=4pt}

\newlist{thmenumi}{enumerate}{10}
\setlist[thmenumi]{leftmargin=0pt, labelindent=\parindent, listparindent=\parindent, labelwidth=0pt, itemindent=!}

\begin{document}

\title{Strong transience for one-dimensional Markov chains with asymptotically zero drifts}
\author{Chak Hei Lo\footnote{Department of Statistical Science, University College London, Gower Street, London WC1E 6BT, UK. Email: \href{mailto:chak.lo@ucl.ac.uk}{\texttt{chak.lo@ucl.ac.uk}}.} \and Mikhail V.\ Menshikov\footnote{Department of Mathematical Sciences, Durham University, Upper Mountjoy, Durham DH1 3LE, UK. Email: \href{mailto:mikhail.menshikov@durham.ac.uk}{\texttt{mikhail.menshikov@durham.ac.uk}}, \href{mailto:andrew.wade@durham.ac.uk}{\texttt{andrew.wade@durham.ac.uk}}.} \and Andrew R.\ Wade\footnotemark[2]}

\maketitle

\begin{abstract}
For near-critical, transient Markov chains on the non-negative integers in the Lamperti regime, where the mean drift at $x$ decays as $1/x$ as $x \to \infty$, we quantify degree of transience via existence of moments for conditional return times and for last exit times, assuming increments are uniformly bounded. Our proof uses a Doob $h$-transform, for the transient process conditioned to return, and we show that the conditioned process is also of Lamperti type with appropriately transformed parameters. To do so, we obtain an asymptotic expansion for the ratio of two return probabilities, evaluated at two nearby starting points; a consequence of this is that the return probability for the transient Lamperti process is a regularly-varying function of the starting point.
\end{abstract}

\medskip


\medskip


\section{Introduction}
\label{sec:intro}

A transient, irreducible Markov chain~$X= (X_0, X_1, \ldots)$ on a countable state space~$S$ has $\Pr ( \tau = \infty ) \in (0,1)$, where $\tau$ is the first return time to a given state. Such a chain is \emph{strong transient} if, moreover, $\Exp [ \tau \mid \tau < \infty ] < \infty$. The concept of strong transience goes back at least to Port~\cite{port67}.
It is known (see Lemma~\ref{lem:equivalence} below) that the strong transience condition $\Exp [ \tau \mid \tau < \infty ] < \infty$
is equivalent both to~(a) $\sum_{n = 1}^\infty n \Pr ( X_n = x ) < \infty$
(for $x \in S$) and to~(b) the last exit time from any state being integrable. For example, in the case where $X$ is simple symmetric random walk on $\mathbb{Z}^d$ started from the origin $X_0 = 0$, one has $\Pr ( X_{2n} = 0) \asymp n^{-d/2}$ and so 
formulation~(a) shows that 
strong transience occurs for $d \geq 5$, while for $d \in \{3,4\}$ there is transience but not strong transience. 

In the present paper we investigate strong transience, and the finer property $\Exp [ \tau^\beta \mid \tau < \infty ] < \infty$, $\beta >0$, in the context of Markov chains on $\ZP := \{0,1,2,\ldots\}$ in the Lamperti regime~\cite{lamp1,lamp2,lamp3}, which is the critical regime for recurrence and transience in the case of processes whose increments have constant-order variance. 
As well as being of interest in their own right, Lamperti processes are prototypical near-critical stochastic processes that have arisen in numerous contexts: see e.g.~\cite{dkw-book} and~\cite{mpw} for surveys.  
Strong transience is of interest both as a quantification of transience, and also as it bears on the geometry of transient trajectories, as has been explored for random walks in the context of 
 the range (that is, how many sites the process has visited by a given time; see~\cite{jo,jp72a,as} and~\cite[\S6.2]{hughes}) and points of self-intersection and cut points (points that separate the past and future trajectories into disjoint sets; see~\cite{de,et}).

\section{Strong transience}
\label{sec:strong-transience}

Before describing in detail Lamperti processes and our main result (in Section~\ref{sec:lamperti} below), we  start by reviewing the concept of strong transience, which is the main focus of this paper, in the context of discrete Markov chains, as in the following assumption.

\begin{description}
\item
[\namedlabel{ass:markov}{M}]
Suppose that $X = (X_n; n \in \ZP)$ is an irreducible, time-homogeneous Markov chain on a countable state space $S$. 
\end{description}

The transition probabilities of $X$ specify a collection of laws $\Pr_x$, $x \in S$, for the chain started from a fixed initial state $x \in S$, i.e., $\Pr_x ( X_0 = x) =1$. We write $\Exp_x$ for the  expectation corresponding to $\Pr_x$. The initial state will play no part in our results, but will be used in formulating our notation. One can also realise $X$ on an externally specified probability space in which $X_0$ is random;  we write simply $\Pr$ for the probability measure in such cases, and we also use $\Pr$ to stand in for $\Pr_x$ when $x$ is unimportant. 
We use the notation $\N := \{1,2,3,\ldots\}$. 

For $y \in S$, 
define the \emph{first hitting time} $\tau_y$ to $y$ and the
\emph{last exit time} $\lambda_y$ from $y$ as
\begin{equation}
    \label{eq:tau-lambda} \tau_y := \inf \{ n \in \N : X_n = y \}; ~~~ \lambda_y := \sup \{ n \in \N : X_n = y \} ,
    \end{equation}
where the conventions $\inf \emptyset := +\infty$ and $\sup \emptyset := 0$ are in force. Of course, the distributions
of $\tau_y$ and $\lambda_y$ depend not only on~$y$ but also on the distribution of~$X_0$; in particular,
we will introduce notation for their moments under $\Pr_x$, when $X_0 = x$ is fixed. 
Define for $\beta > 0$ and $x, y \in S$ the quantities
\begin{align*}
T_\beta (x, y) & :=  \Exp_x \bigl[ \tau^\beta_y \1 { \tau_y < \infty } \bigr] ; \\
L_\beta (x, y) & := \Exp_x \bigl[ \lambda_y^\beta \bigr] ; \\
U_\beta (x,y) & := \Exp_x \sum_{n \in \N} n^\beta \1 { X_n = y }  ; \end{align*}
in the definition of $T_\beta$ the convention is that $\tau_y^\beta \1 { \tau_y < \infty } := 0$ if $\tau_y = \infty$. 
Also write
\[ T_\beta (x) := T_\beta (x,x), ~~~ L_\beta (x) := L_\beta (x,x), ~~~ U_\beta (x) := U_\beta (x, x) ;\]
in words, $T_\beta(x)$ is, for the Markov chain started from $X_0=x$,  the $\beta$th moment of the random variable that is equal to the first return time to~$x$ when finite, and takes value~$0$ if the Markov chain never returns to~$x$.

\begin{definition}
Suppose that~\eqref{ass:markov} holds.
Let $\beta >0$. We say that $X$ is \emph{$\beta$-strong transient} if $T_\beta (x) < \infty$ for some $x \in S$.
If $\beta=1$, we simply say $X$ is \emph{strong transient}.
\end{definition}

The following result gives equivalent formulations of $\beta$-strong transience. The intuition is that, in the transient case, 
the number of returns to a given state is geometrically distributed, so
roughly equivalent tail behaviour is exhibited by  conditional first-return times, last exit times, and even the sum of times at which the process visits that site. 

\begin{lemma}
\label{lem:equivalence}
Suppose that~\eqref{ass:markov} holds, and that $X$ is transient.
Let $\beta >0$. Then the following are equivalent.
\begin{enumerate}[label=(\roman*)]
\setlength\itemsep{-0.2em}
\item\label{lem:equivalence-i}
$T_\beta (x) < \infty$ for some $x  \in S$.
\item\label{lem:equivalence-ii}
$T_\beta (x,y) < \infty$ for all $x,y \in S$.
\item\label{lem:equivalence-iii}
$L_\beta (x) < \infty$ for some $x \in S$.
\item\label{lem:equivalence-iv}
$L_\beta (x,y) < \infty$ for all $x, y \in S$. 
\item\label{lem:equivalence-v}
$U_\beta (x  ) < \infty$ for some $x \in S$.
\item\label{lem:equivalence-vi}
$U_\beta (x,y) < \infty$ for all $x,y \in S$.
\end{enumerate}
\end{lemma}

For $\beta =1$, we can trace the definition of strong transience back to Port~\cite{port67} and Jain and Orey~\cite{jo}
(observing that Lemma~\ref{lem:equivalence} gives equivalence of definitions). Much of Lemma~\ref{lem:equivalence} is contained in Theorem~3.2 of~\cite{port66} (which deals with $\beta=1$)
and Theorem~4 of~\cite{takeuchi}: we give a short proof in Appendix~\ref{sec:appendix}. 
 Note that
$T_1  (x,y)  =  \sum_{n \in \ZP}  \Pr_x ( n < \tau_y < \infty )$ and, by Fubini's theorem,
\begin{align*} U_\beta (x, y) & = 
\Exp_x \sum_{n \in \N} \sum_{m=0}^{n-1} n^{\beta-1} \1 { X_n = y } = \sum_{m \in\ZP} \Exp_x \sum_{n > m} n^{\beta-1} \1 { X_n = y } , \end{align*}
which (for $\beta=1$) is  the form of $U_1(x,y)$ used in~\cite{jo,port67}.

 \section{Lamperti processes}
 \label{sec:lamperti}
 
 We now turn to the class of models that we will study. We suppose that~\eqref{ass:markov} holds for state space $S = \ZP$. 
 It follows from irreducibility that~$X$ is non-confined, i.e., $\limsup_{n \to \infty} X_n = \infty$, a.s.; see e.g.~\cite[Cor.~2.1.10]{mpw}.
 If $X$ is recurrent, then $\liminf_{n \to \infty} X_n = 0$, a.s.,
 while if $X$ is transient, then $\lim_{n \to \infty} X_n = \infty$, a.s.~\cite[Lem.~3.6.5]{mpw}. It is the transient case that interests us here.
 
 The assumptions that we impose  all pertain to the transition probabilities of the Markov chain; the distribution of $X_0$ plays no role.  We will assume the following.
\begin{description}
\item
[\namedlabel{ass:bounded-jumps}{B}]
Suppose that there exists a constant $B \in \N$ such that, for all $n \in \ZP$,
\[
\Pr ( | X_{n+1} - X_n | \leq B   ) =1 . \]  
\item
[\namedlabel{ass:irreducibility}{I}]
Suppose that there exist $\eps>0$ and $m \in \N$ such that, for every $i \in \ZP$,
\begin{equation}
    \max_{1 \leq n \leq m} \Pr_i ( X_n = j ) \geq \eps, \text{ for all } j \in \ZP \text{ with } | j - i | \leq B ,
\end{equation}
where $B$ is the constant from~\eqref{ass:bounded-jumps}.
\end{description}
Assumption~\eqref{ass:bounded-jumps} is boundedness of increments,
while~\eqref{ass:irreducibility} is a strengthening of the irreducibility condition to incorporate some uniformity. 
Indeed, irreducibility shows that for every $i, j \in \ZP$, $\Pr_j ( X_{n_{i,j}} = i ) \geq \eps_{i,j}$ for some $n_{i,j} \in \N$ and $\eps_{i,j} >0$. Hence, setting $m_i := \max_{j:|j-i| \leq B} n_{i,j}$ and $\eps_i := \min_{j:|j-i| \leq B} \eps_{i,j}$ we have $\max_{1 \leq n \leq m_i} \Pr_j (X_n = i ) \geq \eps_i$ for all~$i,j \in \ZP$ with $|j-i| \leq B$, where $m_i \in \N$ and $\eps_i >0$; assumption~\eqref{ass:irreducibility} demands that one may take $m_i$ and $\eps_i$ to be independent of~$i$.

Under assumption~\eqref{ass:bounded-jumps}, the increment moment function $\mu_k : \ZP \to \R$ given by
\begin{equation}
    \label{eq:mu-k-def}
    \mu_k (x) := \Exp [ (X_{n+1} - X_n )^k \mid X_n = x] , \text{ for } x \in \ZP,
\end{equation}
is well defined for any $k \in \N$ (and does not depend on~$n \in \ZP$).

From the point of view of the recurrence and transience of $X$, the most interesting regime is when $2 x \mu_1 (x)$ and $\mu_2 (x)$ are comparable: investigation of this case goes back to pioneering work of Lamperti~\cite{lamp1,lamp2,lamp3}. We will assume the following Lamperti-type asymptotic conditions.

\begin{description}
\item
[\namedlabel{ass:lamperti-transience}{L}]
Suppose that there exist finite constants $c$ and $s^2$ such that $2c > s^2 > 0$, and
\[
 \mu_1 (x) = \frac{c}{x} + o \left( \frac{1}{x \log x} \right) , \text{ and } \mu_2 (x)  = s^2 + o \left( \frac{1}{\log x} \right) , \text{ as }  x \to \infty. \]  
\end{description}
 
 Assuming the preceding assumptions in this section, condition~\eqref{ass:lamperti-transience} ensures transience, via a result of Lamperti~\cite[Thm.~3.1, p.~320]{lamp1}, and this is optimal in the sense that $2c \leq s^2$ implies recurrence~\cite[Thm.~3.5.2, p.~108]{mpw} (see also~\cite{mai}). The main result of this paper is the following classification of $\beta$-strong transience for Lamperti processes of the type described in the present section.
 
 \begin{theorem}
\label{thm:lamp-st}
Suppose that~\eqref{ass:markov} holds for $S= \ZP$, and that~\eqref{ass:bounded-jumps}, \eqref{ass:irreducibility}, and~\eqref{ass:lamperti-transience} hold. Then $X$ is $\beta$-strong transient
if $2c > (2\beta+1) s^2$ and not $\beta$-strong transient if $2c < (2\beta+1) s^2$.
\end{theorem}

Probably the simplest example to illustrate Theorem~\ref{thm:lamp-st}
is that where $X$ is an irreducible nearest-neighbour random walk (or \emph{birth-and-death} chain)
for which
\[ \Pr ( X_{n+1} = x \pm 1 \mid X_n = x ) = \frac{1}{2} \pm \frac{c}{2x} , \text{ for all } x \in \N \text{ with } x > |c|. \]
For this example, $\mu_1 (x) = c/x$ and $\mu_2 (x) = 1$, so~\eqref{ass:lamperti-transience}
holds with $s^2 = 1$ provided $c > 1/2$ (the transient case). Theorem~\ref{thm:lamp-st}
says that $X$ is $\beta$-strong transient
if $2c > 2\beta +1$ and not $\beta$-strong transient if $2c < 2\beta +1$.
While the assumption~\eqref{ass:bounded-jumps}
is an obstacle to some applications (see Remark~\ref{rems:lamp-st}\ref{rems:lamp-st-d}),
Lamperti processes with bounded jumps have found recent applications
in  survival and dominance of agents trading in 
financial markets, and
in polymer pinning and wetting models from statistical physics
(see, respectively, \cite{bdindo} and~\cite{alexander}, and references therein).

Before making some more detailed comments on Theorem~\ref{thm:lamp-st}
(in Remarks~\ref{rems:lamp-st} below), we outline the scheme of proof, around which the rest of this paper is constructed. 
The first step in the proof of Theorem~\ref{thm:lamp-st} is to define, via an appropriate Doob $h$-transform, a version of the Markov chain~$X$ conditioned to return to~$0$. This construction is presented in 
Section~\ref{sec:doob}, and is based on a modulation of the transition
probabilities according to the function $h(x)$, which is the probability that the (transient) Markov chain $X$ visits $0$ in finite time, started from $x \in S$. Analysis of the behaviour of the conditioned Markov chain requires a rather precise study of the hitting-probability function~$h$. There are two main component results in this direction in the body of the paper, which are of some independent interest, but whose formal statement we defer until later, after the necessary definitions and notation.
\begin{itemize}
    \item Theorem~\ref{thm:h-estimate} presents an asymptotic estimate for $h(x+z)/h(x)$ for fixed $z$ as $x \to \infty$. 
    \item Theorem~\ref{thm:conditional-moments} establishes that the conditioned process is itself a sort of Lamperti process, satisfying a version of~\eqref{ass:lamperti-transience} 
    with appropriately transformed increment moment parameters, but losing a factor of $\log x$ in the error terms. 
\end{itemize}
This loss of control in the error terms for the conditioned process is the reason for the fact that Theorem~\ref{thm:lamp-st} does not cover the boundary case where $2c = (2 \beta+1) s^2$: see also Remark~\ref{rems:lamp-st}\ref{rems:lamp-st-a}.

 The proof of Theorem~\ref{thm:lamp-st}, which is presented at the end of Section~\ref{sec:doob}, follows from Theorem~\ref{thm:conditional-moments} and estimates on passage-time moments for Lamperti processes from~\cite{aim,mpw}.
   Theorem~\ref{thm:h-estimate}, whose proof is presented in Section~\ref{sec:proofs} and contains most of the technical work of the paper, yields Theorem~\ref{thm:conditional-moments}. (A
   corollary of the ratio asymptotic  in  Theorem~\ref{thm:h-estimate} is the regular variation of the function~$h$.) 
   Section~\ref{sec:discussion}  discusses some additional context and related literature. In particular, Section~\ref{sec:random-walks} presents some comparisons between Theorem~\ref{thm:lamp-st} and strong transience results for multidimensional, homogeneous random walks; this comparison gives access to some additional intuition for the point of the phase transition exhibited in Theorem~\ref{thm:lamp-st}.
   Section~\ref{sec:branching} describes a connection between Lamperti processes and branching processes with migration,
   in which context we discuss work of Kosygina \& Zerner~\cite{kz} which complements the present work (see also Remark~\ref{rems:lamp-st}\ref{rems:lamp-st-z}). 
   Finally, Appendix~\ref{sec:appendix} gives the proof of the general Lemma~\ref{lem:equivalence} on characterization of strong transience for countable Markov chains.

\begin{remarks}
\label{rems:lamp-st}
\begin{myenumi}[label=(\alph*)]
\setlength{\itemsep}{0pt plus 1pt}
\item\label{rems:lamp-st-a}
As mentioned above, Theorem~\ref{thm:lamp-st} does not cover the boundary case where $2c = (2 \beta+1) s^2$.
Such boundary cases can be  delicate (see e.g.~\cite{mai}). A central step in the proof we present below is a coupling argument, which seems to
inherently lead to a gain of an $O (\log x)$ factor in the error terms of the increment moments of the conditioned
process compared to $\mu_1, \mu_2$ in~\eqref{ass:lamperti-transience} (cf.~Theorem~\ref{thm:conditional-moments}).
There are two cases where we believe
that  Theorem~\ref{thm:lamp-st} can be strengthened to assert that the boundary case where $2c = (2 \beta+1) s^2$ is \emph{not} $\beta$-strong transient.
The first would be to impose a stronger version of~\eqref{ass:lamperti-transience} in which the $\log x$ factor in the denominators in the $o( \, \cdot \, )$ terms is replaced by $\log^2 x$.
The second is the case of a nearest-neighbour walk, since there one can avoid the coupling step in the proof below altogether: see Remark~\ref{rems:ratio}\ref{rems:ratio-c}. In those two cases, Theorem~\ref{thm:conditional-moments} could be strengthened to achieve error terms comparable to~\eqref{ass:lamperti-transience}, and hence one could settle the boundary case where $2c = (2 \beta+1) s^2$. We do expect that the boundary case being \emph{not} strong transient is rather generic, but it is not clear to us whether the assumption~\eqref{ass:lamperti-transience} is sufficient for this in general.
\item\label{rems:lamp-st-b}
By considering $U_\beta (0) = \sum_{n \in \N} n^\beta \Pr ( X_n =0 )$,
Theorem~\ref{thm:lamp-st} would also follow from Lemma~\ref{lem:equivalence} if one established the validity of the local limit estimate
\begin{equation}
    \label{eq:llt}
 \lim_{n \to \infty} \frac{\log \Pr ( X_n = 0 )}{\log n} = - \frac{2c+s^2}{2s^2} .\end{equation}
However, the only local limit theorem results that we are aware of
assume nearest-neighbour increments, where $X_{n+1} - X_n = \pm 1$, and null recurrence.
In that setting, $s^2 =1$, and the limit theorem~\eqref{eq:llt}
is proved for $|2c| < s^2$ by Alexander~\cite[eq.~(2.19)]{alexander},
with earlier work by
Rosenkrantz~\cite[Thm.~1]{rosenkrantz} 
(in the special case $c=1/2$)
and Fal$'$~\cite[Thm.~1]{fal}.
Rosenkrantz~\cite{rosenkrantz} and Fal$'$~\cite{fal} use the Karlin--McGregor spectral representation and require a specific form for the transition probabilities, while Alexander~\cite{alexander} uses an excursion approach, which seems to not be applicable in the transient case. We are not aware of~\eqref{eq:llt} having been  established in any transient example. An heuristic explanation for~\eqref{eq:llt}
is that $\Pr ( X_n =0 )$ should be comparable to $\Pr ( \tau_0 = n )$ and $\Pr (\lambda_0 = n)$, which
Theorem~\ref{thm:lamp-st} suggests are about $n^{-q+o(1)}$ for $q = \frac{2c+s^2}{2s^2}$;
see~\cite{bd,dk} for comparison results for such quantities in the case of  some transient random walks with i.i.d.~increments.
\item\label{rems:lamp-st-z}
A study of excited random walk due to Kosygina \& Zerner~\cite{kz}
establishes criteria for strong transience for a class of
\emph{branching processes with migration},
which can be cast as Lamperti processes (see Section~\ref{sec:branching} below).
The technical approach of~\cite{kz}, like ours, uses a Doob $h$-transform,
and recognizes that conditioning retains the Lamperti-type character of the process;
on the other hand, the main line of the analysis of~\cite{kz} goes via a diffusion approximation, quite different to our method. Furthermore, in~\cite{kz} are obtained some asymptotic results on the harmonic function~$h$ that complement our Theorem~\ref{thm:h-estimate} below; see Remark~\ref{rems:ratio}\ref{rems:ratio-a}.
\item\label{rems:lamp-st-c}
A natural continuum comparison with the Lamperti process in Theorem~\ref{thm:lamp-st}
is a Bessel process with ``dimension'' (parameter) $\delta = (2c+s^2)/s^2$.
When $\delta >0$, this comparison is formalized at the level of weak convergence by classical work of Lamperti~\cite{lamp2}; see also~\cite[\S 3.11]{mpw}. The transient case is $\delta > 2$, and the fact that the Bessel process has Gamma marginals
establishes a Bessel analogue of~\eqref{eq:llt}. However, local approximation of Lamperti processes by Bessel processes has only been carried through in the nearest-neighbour case, as far as we are aware: see, e.g.~\cite{alexander}.
\item\label{rems:lamp-st-d}
It would be interesting to relax the assumption~\eqref{ass:bounded-jumps} of bounded jumps,
as this would broaden the range of applications
to include, for example, random walks on a half-strip~\cite{gw}
and branching processes with migration (see Section~\ref{sec:branching}).
We emphasize, however,
that the technical challenges in our approach that lead to the ``extra $\log$'' mentioned in
remark~\ref{rems:lamp-st-a} are likely to be more demanding with unbounded jumps.
In particular, the method of the present paper relies heavily on at least the \emph{lower} boundedness of jumps for the exponential conditional mixing estimate in Proposition~\ref{prop:coupling} below; elsewhere (for example in the Lyapunov function estimates) one should be able to assume boundedness of certain increment moments instead.
\item\label{rems:lamp-st-e}
There are other ways to quantify transience, but those that capture the fact that the process is \emph{diffusive} do not show the phase transition that appears in Theorem~\ref{thm:lamp-st}. 
The following remarks apply under the hypotheses of Theorem~\ref{thm:lamp-st} (and, in some cases, additional assumptions). 
Diffusive weak convergence results and almost-sure growth bounds that show that $X_n$ is typically about $n^{1/2}$, building on early work of Lamperti~\cite{lamp2}, can be found in~\cite[\S\S 3.9--3.11]{mpw}; iterated-logarithm results are given e.g.~in~\cite{gallardo}. Moreover, for the renewal function
$H(x) := \Exp \sum_{n=0}^\infty \1 { X_n \leq x}$,
Theorem~5 of Denisov \emph{et al.}~\cite{dkw1}
shows that $\lim_{x \to \infty} x^{-2} H(x) = \frac{1}{2c -s^2}$; see also~\cite{dkw2} for some finer results. Another way to quantify transience is via the number of \emph{cutpoints}~\cite{de,lmw}; it follows from~\cite[Thm.~1.2]{lmw} that $X$ has infinitely many cutpoints, a.s.
\item\label{rems:lamp-st-f}
Theorem~\ref{thm:lamp-st}
 shows that 
 a small correction is required in  Lemma~3.10.8 of the work~\cite{mpw} involving two of the authors here. Indeed, that result includes the incorrect assertion that the last exit time $\lambda_x$ has $\Exp \lambda_x < \infty$ for every transient Lamperti process. 
The statement of that lemma should be replaced by $\Exp \sum_{n=0}^{\infty} \1 { X_n \leq x } < \infty$ for all $x \geq 0$,
which is what is required for the proof of Theorem~3.10.1 in~\cite{mpw}. The argument in the published proof of Lemma~3.10.8 in~\cite{mpw} can readily be corrected to obtain this.
\end{myenumi}
\end{remarks}

 \section{The conditioned Markov chain}
 \label{sec:doob}
 
 Suppose that~\eqref{ass:markov} holds, and for $i,j \in S$ write 
 $p_{i,j}:= \Pr ( X_{n+1} = j \mid X_n = i ) = \Pr_i (X_1 = j)$
 for the one-step transition probabilities of $X$. Thus
for all $n \in \N$ and all $i, x_1, \ldots, x_n \in S$,
 \begin{equation}
     \label{eq:markov-law}
 \Pr_i ( X_1 = x_1, \ldots, X_n = x_n ) = p_{i,x_1} p_{x_1,x_2} \cdots p_{x_{n-1},x_n} . \end{equation}
Distinguish an arbitrary state in $S$ by $0 \in S$, and write $\tau := \inf \{ n \in \ZP : X_n = 0\}$. 
Define the hitting probabilities
 \begin{equation}
     \label{eq:h-def}  h(i) := \Pr_i ( \tau < \infty ) , \text{ for } i \in S .
     \end{equation}
 Note that $h(0) = 1$ and, in the notation at~\eqref{eq:tau-lambda}, $h(i) = \Pr_i (\tau_0 < \infty)$ for $i \in S \setminus \{0\}$.
By irreducibility, $h(i) >0$ for all $i \in S$, while, if $X$ is transient, then $h(i) < 1$ for infinitely many $i \in S \setminus \{0\}$.

 Define $\tp_{0,j} := p_{0,j}$ for all $j \in S$, and
 \begin{equation}
     \label{eq:tp}
 \tp_{i,j} := \frac{h(j)}{h(i)} p_{i,j}, \text{ for all } i \in S, j \in S \setminus \{0\}.
  \end{equation}
The function $h: S \to (0,1]$ is harmonic for $P := (p_{i,j})_{i,j \in S}$
stopped at~$0$, in the sense that
\[ h(i) = \Exp_i \sum_{j \in S} \1{ X_1 =j, \, \tau < \infty } = \sum_{j \in S} p_{i,j} h(j) , \text{ for all } i \in S \setminus \{0\}, \]
as follows from the Markov property. 
Thus $\sum_{j \in S} \tp_{i,j} =1$ for all $i \in S$,
so the $\tp_{i,j}$ define a Markov transition law with modified transition probabilities when away from~$0$.

Let $\tPr_i$ denote the law on $X$ generated by initial state $i$ and Markov transition probabilities $\tp_{i,j}$,
 i.e., for all $n \in \N$ and all $i, x_1, \ldots, x_n \in S$,
 \begin{align}
     \label{eq:doob-law}
    & \tPr_i ( X_0 =i, X_1 = x_1, \ldots, X_n = x_n )  = \tp_{i,x_1} \tp_{x_1,x_2} \cdots \tp_{x_{n-1},x_n} .\end{align}
    In particular, it follows from~\eqref{eq:doob-law}, \eqref{eq:tp} and~\eqref{eq:markov-law} that, for all $i, x_1, \ldots, x_{n-1} \in S \setminus \{ 0\}$, 
    \begin{align}
     \label{eq:doob-law2}
     & \tPr_i ( X_0 =i, X_1 = x_1, \ldots, X_n = x_n )   \nonumber\\
     & {} \qquad {} = \Pr_i ( X_0 =i, X_1 = x_1, \ldots, X_n = x_n ) \frac{h(x_n)}{h(i)} \text{ for } x_n \in S.
     \end{align}
  Let $\tExp_i$ denote the expectation corresponding to $\tPr_i$.
 The Markov law given by~\eqref{eq:doob-law} is the Doob $h$-transform
 of the law~\eqref{eq:markov-law} corresponding to $h$ given by~\eqref{eq:h-def}; it has the following conditioning interpretation.
 
 \begin{lemma}
 \label{lem:doob-h-transform}
  Suppose that~\eqref{ass:markov} holds. 
  Then $\tPr_i (\tau < \infty) = 1$ for all $i \in S$. Moreover, for every $i \in S$, under $\Pr_i$, the law of $(X_0, X_1, \ldots, X_\tau)$ given $\tau < \infty$
 is the same as the law of $(X_0, X_1, \ldots, X_\tau)$ under $\tPr_i$.
 \end{lemma}
 
 Lemma~\ref{lem:doob-h-transform} allows us to interpret $\beta$-strong transience for $X$ under the original measure~$\Pr_i$ in terms of expected return times under the transformed measure~$\tPr_i$, as expressed in the following corollary.
 
 \begin{corollary}
 \label{cor:conditioning}
   Let $\beta >0$ and $i \in S$. 
 Then $\Exp_i [ \tau^\beta \1 { \tau < \infty }  ] = h(i) \tExp_i [ \tau^\beta ]$.
 \end{corollary}
 
 \begin{proof}[Proof of Lemma~\ref{lem:doob-h-transform}.]
Clearly $\tPr_0 ( \tau < \infty) =1$. For $i \in S \setminus \{0\}$,
\begin{align*}
    \tPr_i ( \tau < \infty) & = \sum_{n \in \N} \sum_{x_1, \ldots, x_{n-1} \in S \setminus \{0\} } \tPr_i ( X_1 = x_1, \ldots, X_{n-1} = x_{n-1}, X_n = 0 ) \\
    & = \frac{1}{h(i)} \sum_{n \in \N} \sum_{x_1, \ldots, x_{n-1} \in S \setminus \{0\} } \Pr_i ( X_1 = x_1, \ldots, X_{n-1} = x_{n-1}, X_n = 0 ),
\end{align*}
using~\eqref{eq:doob-law2} and the fact that $h(0)=1$. But for $i \in S \setminus \{0\}$,
\[   \sum_{n \in \N} \sum_{x_1, \ldots, x_{n-1} \in S \setminus \{0\} } \Pr_i ( X_1 = x_1, \ldots, X_{n-1} = x_{n-1}, X_n = 0 ) = \Pr_i ( \tau < \infty ) = h(i) ,\]
by~\eqref{eq:h-def}, 
and so we have shown that $\tPr_i (\tau < \infty) = 1$. For the final statement in the lemma, it suffices to prove that, for any $n \in \ZP$ and any $x_1, \ldots, x_n \in S \setminus \{0\}$,
 \begin{align}
 \label{eq:doob-probs-1}
     \Pr_i ( \cap_{i=1}^n \{ X_i = x_i \} \mid \tau < \infty ) & = \tPr_i ( \cap_{i=1}^n \{ X_i = x_i \} ) ; \\
 \label{eq:doob-probs-2}
     \Pr_i ( \cap_{i=1}^n \{ X_i = x_i \} \cap \{ X_{n+1} = 0 \} \mid \tau < \infty ) & = \tPr_i ( \cap_{i=1}^n \{ X_i = x_i \} \cap \{ X_{n+1} = 0 \} ) .
 \end{align}
 Here we have that, by the strong Markov property under $\Pr_i$,
 for $x_1, \ldots, x_n \neq 0$, 
 \begin{align*}
          h(i) \Pr_i ( X_1 = x_1, \ldots, X_n = x_n \mid \tau < \infty ) & = \Pr_i ( X_1 = x_1, \ldots, X_n = x_n , \tau < \infty ) \\
& = \Pr_i ( X_1 = x_1, \ldots, X_n = x_n ) h (x_n)  \\
          & = h(i) \tp_{i,x_1} \cdots \tp_{x_{n-1},x_n} , 
 \end{align*}
 by~\eqref{eq:doob-law} and~\eqref{eq:doob-law2}, 
and this yields~\eqref{eq:doob-probs-1}. Similarly,
  \begin{align*}
          h(i) \Pr_i ( X_1 = x_1, \ldots, X_n = x_n, X_{n+1} = 0 \mid \tau < \infty ) & = \Pr_i ( X_1 = x_1, \ldots, X_n = x_n , X_{n+1} = 0 ) \\
          & = h(i) \tp_{i,x_1} \cdots \tp_{x_{n-1},x_n} \tp_{x_n,0}, 
 \end{align*}
 using the fact that $h(0) = 1$, giving~\eqref{eq:doob-probs-2}.
 \end{proof}
 
 Now we return to the case where $X$ is a Lamperti process on $S = \ZP$,
 as in Section~\ref{sec:lamperti}. 
 In view of Corollary~\ref{cor:conditioning}, we will study the $\beta$-strong transience of $X$ under $\Pr_i$ by studying the 
  conditioned version of $X$ under $\tPr_i$. 
 Since $h(i) >0$ for all $i \in \ZP$, it follows from~\eqref{eq:tp}
 that $\tp_{i,j} >0$ whenever $p_{i,j} >0$, and hence
  irreducibility under $p_{i,j}$ implies irreducibility under $\tp_{i,j}$. 
Similarly, under condition~\eqref{ass:bounded-jumps},
 it is the case that $\tPr_x ( | X_1 - X_0 | \leq B ) =1$
 also holds for the conditioned process. 
   For $k \in \N$, define
 the increment moment functions corresponding to the conditioned process by
 \begin{equation}
    \label{eq:tmu-k-def}
    \tmu_k (x) := \tExp_x [ (X_{1} - X_0 )^k ] = \sum_{z =- x \wedge B}^B \frac{z^k h(x+z) p_{x,x+z}}{h(x)}, \text{ for } x \in \ZP.
\end{equation}
Recall from Corollary~\ref{cor:conditioning} that $\beta$-strong transience of the $\Pr$ process equates to existence of $\beta$-moments of return times for the $\tPr$ process.
The main part of the proof of Theorem~\ref{thm:lamp-st}
 will be provided by Theorem~\ref{thm:conditional-moments} below,
 which shows that the $\tPr$ process is itself of Lamperti type,
 but with transformed parameters (such that it is recurrent, rather than transient),
 and with weaker control of the error terms compared to~\eqref{ass:lamperti-transience}.

\begin{theorem}
\label{thm:conditional-moments}
Suppose that~\eqref{ass:markov} holds for $S= \ZP$, and that~\eqref{ass:bounded-jumps}, \eqref{ass:irreducibility}, and~\eqref{ass:lamperti-transience} hold. 
Then $\tPr := ( \tPr_x )_{x \in \ZP}$ defines an irreducible Markov chain on $\ZP$ with uniformly bounded increments, for which
\begin{equation}
    \label{eq:tmu-asymptotics}
 \tmu_1 (x) = \frac{s^2-c+o(1)}{x} , \text{ and } \tmu_2 (x)  = s^2 + o ( 1  ) , \text{ as }  x \to \infty. \end{equation}
\end{theorem}
 \begin{remark}
  From~\eqref{eq:tmu-asymptotics} we have that $\lim_{x \to \infty} ( 2 x  \tmu_1 (x) - \tmu_2 (x) ) = s^2 - 2 c$. 
 Note that the quantity $s^2 -2c$ is the negative
 of the corresponding quantity for the unconditioned process, namely
 $\lim_{x \to \infty} ( 2 x  \mu_1 (x) - \mu_2 (x) ) = 2c -s^2$.
This sign change exactly agrees with the Lamperti phase transition~\cite[Thm.~3.1, p.~320]{lamp1}, since the process under $\Pr$ is transient and under $\tPr$ is recurrent. The critical case $2c=s^2$ is null recurrent under~\eqref{ass:lamperti-transience}. 
 \end{remark}
 
 We prove Theorem~\ref{thm:conditional-moments} in Section~\ref{sec:proofs};
 taking Theorem~\ref{thm:conditional-moments} as given for now, we can complete the proof of Theorem~\ref{thm:lamp-st}.
 
  \begin{proof}[Proof of Theorem~\ref{thm:lamp-st}.]
 By Corollary~\ref{cor:conditioning} and Lemma~\ref{lem:equivalence}, we have that $X$ is $\beta$-strong transient if and only if $\tExp_i [ \tau^\beta ] < \infty$ for some (hence every) $i \in S \setminus \{0\}$.
 Theorem~3.2.6 of~\cite{mpw} (which  specializes to the case of
  Markov chains on $\ZP$ results of~\cite{aim}) shows that sufficient for $\tExp_i [ \tau^\beta ] < \infty$ is
$ 2a + (2\beta-1) b < 0$,
where $a := \lim_{x\to\infty} x \tmu_1 (x)$
and $b: = \lim_{x \to \infty} \tmu_2 (x)$,
while sufficient for $\tExp_i [ \tau^\beta ] = \infty$
is $2a + (2\beta-1) b > 0$.
With Theorem~\ref{thm:conditional-moments}, this means we have $\beta$-strong transience if $2 (s^2 -c) + (2\beta -1) s^2 < 0$, and we do not have $\beta$-strong transience if $2 (s^2 -c) + (2\beta -1) s^2 > 0$. This establishes Theorem~\ref{thm:lamp-st}.
 \end{proof}

 \section{Ratio expansions for return probabilities}
 \label{sec:proofs}
 
 The main aim of this section is to prove Theorem~\ref{thm:conditional-moments}; to do so we study in more detail the return-probability function~$h$ defined at~\eqref{eq:h-def}. 
In order to estimate $\tmu_1(x)$ and $\tmu_2(x)$ in Theorem~\ref{thm:conditional-moments}, we need in~\eqref{eq:tmu-k-def} to have estimates for $h(x+z)/h(x)$, at least for $|z| \leq B$.
Theorem~\ref{thm:h-estimate} below presents the central estimate.
As a consequence, we deduce (in Corollary~\ref{cor:reg-var} below) that the function~$h$ is regularly varying and eventually decreasing, which is a result of some independent interest.

For $c, s^2$ the parameters in~\eqref{ass:lamperti-transience}, define the critical exponent
\begin{equation}
    \label{eq:gamma-c}
    \gcr := \frac{2c -s^2}{s^2}.
\end{equation}
Note that, by~\eqref{ass:lamperti-transience}, we have $\gcr \in (0,\infty)$.

 \begin{theorem}
 \label{thm:h-estimate}
 Suppose that~\eqref{ass:markov} holds for $S= \ZP$, and that~\eqref{ass:bounded-jumps}, \eqref{ass:irreducibility}, and~\eqref{ass:lamperti-transience} hold. 
 We have, uniformly in $|z| \leq B$, as $x \to \infty$, 
 \begin{equation}
     \label{eq:h-ratio}
 \frac{h(x+z)}{h(x)} = 1 - \frac{\gcr z}{x} + o (x^{-1} ) . \end{equation}
 \end{theorem}

\begin{corollary}
\label{cor:reg-var}
 Suppose that~\eqref{ass:markov} holds for $S= \ZP$, and that~\eqref{ass:bounded-jumps}, \eqref{ass:irreducibility}, and~\eqref{ass:lamperti-transience} hold. Then $h(x+1) < h(x)$ for all sufficiently large $x \in \ZP$, and there exists a function $L: \N \to (0,\infty)$, slowly varying at $\infty$, such that $h(x) = x^{-\gcr} L(x)$ for all $x \in \N$.
\end{corollary}

\begin{remarks}\phantomsection
\label{rems:ratio}
\begin{myenumi}[label=(\alph*)]
\setlength{\itemsep}{0pt plus 1pt}
\item\label{rems:ratio-a}
Kosygina \& Zerner~\cite{kz} obtain an asymptotic estimate for~$h$, of a somewhat different kind than that in Theorem~\ref{thm:h-estimate}, 
for a class of branching processes with migration; see Section~\ref{sec:branching} below for the translation to the Lamperti context. The assumptions in~\cite{kz} require specific structure for the process, and specified distributions of increments, but they do not (and must not) assume  bounded jumps. In our notation, Proposition~4.3 of~\cite{kz} says that $\lim_{x \to \infty} x^{\gcr} h(x)$ is a finite positive constant, and hence
\begin{equation}
    \label{eq:kz-asyptotics}
    \lim_{x \to \infty} \frac{h(\lfloor \lambda x  \rfloor)}{h(x)} = \lambda^{-\gcr}, \text{ for any } \lambda \in (0,\infty).
\end{equation}
This statement neither implies, nor is implied by, the ratio limit result~\eqref{eq:h-ratio}; it does imply regular variation of~$h$. An anonymous referee suggests the terminology that~\eqref{eq:h-ratio} be described as a \emph{local-at-infinity} result, in contrast to the non-local asymptotics~\eqref{eq:kz-asyptotics}.
\item\label{rems:ratio-b}
Another adjacent result to Corollary~\ref{cor:reg-var} is Theorem~2.20 of Denisov~\emph{et al.}~\cite{dkw-book}, which
provides upper and lower bounds for $h(x)$ in terms of a regularly-varying ``near-harmonic'' function,
under moments conditions weaker than the uniform bound~\eqref{ass:bounded-jumps}.
\item
The slowly-varying function $L$ in Corollary~\ref{cor:reg-var}   cannot be determined by asymptotic assumptions alone: it can change by a constant factor if one modifies a single transition probability near~$0$, for example.
\item\label{rems:ratio-c}
If~\eqref{ass:bounded-jumps} is augmented with the additional \emph{left continuity} (or left \emph{skip free}) assumption that $X_{n+1} - X_n \geq -1$, a.s., then the proof of Theorem~\ref{thm:h-estimate} (and hence Theorem~\ref{thm:lamp-st}) simplifies significantly. Indeed, left continuity implies that
\begin{equation}
    \label{eq:left-continuity}
 h(x+1) = \Pr_{x+1} (\tau_x < \infty ) h (x), \text{ for all } x \in \ZP, \end{equation}
 and in this case a relatively crude optional stopping argument,
 based on the fact that $X_n^{\gcr \pm \eps}$, $\eps>0$,
 is a sub/supermartingale outside a bounded set (cf.~Lemma~\ref{lem:lyapunov-function} below for a finer result) 
 is enough to show that \[ \Pr_{x+1} (\tau_x < \infty) = 1 - \frac{\gcr}{x} + o(x^{-1} ) , \text{ as } x \to \infty, \]
 which, with~\eqref{eq:left-continuity},
 yields~\eqref{eq:h-ratio}. 
\end{myenumi}
\end{remarks}
 
  \begin{proof}[Proof of Corollary~\ref{cor:reg-var}.]
 Since $h(0) = 1$, we can write
 \[ h(x) = \exp \left\{ \sum_{y=0}^{x-1} \log \left( \frac{h (y+1)}{h(y)} \right) \right\} = \exp \left\{ \sum_{y=1}^{x-1} \log \left( 1 - \frac{\gcr}{y} + o (y^{-1}) \right) \right\} ,
 \]
 by the $z=1$ case of Theorem~\ref{thm:h-estimate}. Recalling that $\sum_{y=1}^{x-1} y^{-1} = \log x + \upsilon + o(1)$, where $\upsilon \approx 0.5772$ is Euler's constant, it follows that
 $h(x) = x^{-\gcr} L(x)$, where
 \[ L(x) = \exp \left\{ - \upsilon \gcr + o(1) + \sum_{y=1}^{x-1} o (y^{-1} ) \right\} , \]
 which implies that $L$ is slowly varying~\cite[p.~12]{bgt}. Moreover, from the $z=1$ case of~\eqref{eq:h-ratio}, we have that, for all $x$ sufficiently large,
 \[ \frac{h(x+1)}{h(x)} < 1 - \frac{\gcr}{2x} ,\]
 and hence $h$ is eventually decreasing, as claimed.
 \end{proof}
 
 The rest of this section is devoted to the proof of Theorem~\ref{thm:h-estimate}.
 As described in Remark~\ref{rems:ratio}\ref{rems:ratio-c},
 this is relatively straightforward if the process is left continuous, because there is no randomness in the entrance distribution when crossing a level to the left.
 In general, the argument outlined in Remark~\ref{rems:ratio}\ref{rems:ratio-c}
 does not work as stated, but it must be improved;
 we show (Proposition~\ref{prop:coupling} below) that left-crossing distributions stabilize rapidly, using a coupling argument, and we use a refined Lyapunov function to obtain a sufficient optional stopping estimate (see Lemma~\ref{lem:eta-R-bound} below). We start by introducing the Lyapunov function that we will use.
 
 Define $\RP := [0,\infty)$. 
  For $\gamma \in \RP$ and $\nu \in \R$, define the Lyapunov function $f_{\gamma,\nu} : \RP \to (0,\infty)$ by
 \begin{equation}
     \label{eq:f-def}
 f_{\gamma,\nu} (x) := \begin{cases} x^{-\gamma} \log^\nu x &\text{if } x \geq \re ,\\
 \re^{-\gamma} & \text{if } 0 \leq x < \re .
 \end{cases}
 \end{equation}
 The next lemma demonstrates that $f_{\gamma,\nu}$
 applied to $X$ gives a sub/supermartingale when $X_n$ is outside a bounded set, for appropriate sign of $\nu$ and with $\gamma = \gcr$
 given at~\eqref{eq:gamma-c}.
 
 \begin{lemma}
 \label{lem:lyapunov-function}
  Suppose that~\eqref{ass:markov} holds for $S= \ZP$, and that~\eqref{ass:bounded-jumps} and~\eqref{ass:lamperti-transience} hold. 
 If $\nu >0$, there exists $x_0 \in \RP$ such that
 \[ \Exp \bigl[ f_{\gcr,\nu} (X_{n+1} ) -  f_{\gcr,\nu} (X_{n} ) \bigmid X_n = x \bigr] \leq 0, \text{ for all } x \geq x_0 .\]
 On the other hand,  if $\nu < 0$, there exists $x_0 \in \RP$ such that
 \[ \Exp \bigl[ f_{\gcr,\nu} (X_{n+1} ) -  f_{\gcr,\nu} (X_{n} ) \bigmid X_n = x \bigr] \geq 0, \text{ for all } x \geq x_0 .\]
 \end{lemma}
 \begin{proof}
 Lemma~3.4.1 of~\cite{mpw} shows that (noting the change in sign of $\gamma$ there), for $\gamma \in \R$,
 \begin{align*}
     \Exp \bigl[ f_{\gamma,\nu} (X_{n+1} ) -  f_{\gamma,\nu} (X_{n} ) \bigmid X_n = x \bigr] & =
     - \frac{\gamma}{2x^2} \left( 2 x \mu_1 (x) - (\gamma +1 ) \mu_2 (x) \right) f_{\gamma,\nu} (x) \\
& {} \quad {}     + \frac{\nu}{2 x^2 \log x} \left( 2x \mu_1 (x) - (2 \gamma +1 ) \mu_2 (x) \right)  f_{\gamma,\nu} (x)\\
& {} \quad {} 
     + O ( x^{-2} \log^{-2} x ) f_{\gamma,\nu} (x),
 \end{align*}
 as $x \to \infty$, 
 where $\mu_k$ is as defined at~\eqref{eq:mu-k-def}. 
 By assumption~\eqref{ass:lamperti-transience}, it follows that
 \begin{align*}
     \Exp \bigl[ f_{\gamma,\nu} (X_{n+1} ) -  f_{\gamma,\nu} (X_{n} ) \bigmid X_n = x \bigr] & =
     - \frac{\gamma}{2x^2} \left(2 c - (\gamma + 1 ) s^2 \right) f_{\gamma,\nu} (x) \\
& {} \quad {}     + \frac{\nu}{2 x^2 \log x} \left(2 c - (2 \gamma +1 ) s^2 +o(1) \right)  f_{\gamma,\nu} (x) .
 \end{align*}
 Taking $\gamma = \gcr$ as defined at~\eqref{eq:gamma-c},
 and using the fact that $2c - (\gcr +1)s^2 =0$, we get
 \[  \Exp \bigl[ f_{\gcr,\nu} (X_{n+1} ) -  f_{\gcr,\nu} (X_{n} ) \bigmid X_n = x \bigr] 
 = - \frac{\nu}{2 x^2 \log x} \left(\gcr s^2 +o(1) \right)  f_{\gcr,\nu} (x) ,
 \]
 which, since $\gcr s^2 >0$, yields the result.
 \end{proof}
 
 The next result uses optional stopping to give
 rough bounds on $h(x)$ for later use.
 
 \begin{lemma}
 \label{lem:crude-bound}
  Suppose that~\eqref{ass:markov} holds for $S= \ZP$, and that~\eqref{ass:bounded-jumps} and~\eqref{ass:lamperti-transience} hold. Define $h$ by~\eqref{eq:h-def}. 
 Let $\eps >0$. Then, for all $x$ sufficiently large,
\begin{equation}
    \label{eq:h-crude-bound}
x^{-\gcr} \log^{-\eps} x \leq h(x) \leq x^{-\gcr} \log^\eps x .\end{equation}
 \end{lemma}
 
 \begin{remark}
 The bounds in Lemma~\ref{lem:crude-bound} are too weak to give a good estimate of the key ratio in Theorem~\ref{thm:h-estimate}; indeed, from Lemma~\ref{lem:crude-bound} one cannot even conclude that
 $\lim_{x \to \infty} h(x+z)/h(x)$ exists. Thus the proof of Theorem~\ref{thm:h-estimate} requires additional ideas beyond standard Lyapunov function estimates.
 \end{remark}
 
 \begin{proof}[Proof of Lemma~\ref{lem:crude-bound}.]
Let $\nu >0$ and write $f := f_{\gcr,\nu}$
as defined at~\eqref{eq:f-def} with $\gamma = \gcr$ given by~\eqref{eq:gamma-c}. Take $x_0$ to be the constant in Lemma~\ref{lem:lyapunov-function}. Set $\eta := \inf \{ n \in \ZP: X_n \leq x_0\}$ and
\begin{equation}
    \label{eq:sigma-def}
    \sigma_r := \inf \{ n \in \ZP: X_n \geq r \}, \text{ for } r \geq 0.
\end{equation}
Since $\limsup_{n \to \infty} X_n = \infty$ and, by~\eqref{ass:bounded-jumps}, $X_n \leq X_0 + Bn$, a.s., we have $\sigma_r < \infty$ for all $r \in \RP$, a.s.,
and $\sigma_r \uparrow \infty$ as $r \to \infty$. 
Take $r \in (x_0, \infty)$. Then Lemma~\ref{lem:lyapunov-function} and the fact that $\sup_{x \geq 0} f(x) < \infty$ shows that $f (X_{n \wedge \eta \wedge \sigma_r} )$, $n \in \ZP$, is a non-negative, bounded supermartingale.
Moreover, 
$\lim_{n \to \infty} f (X_{n \wedge \eta \wedge \sigma_r} ) = f ( X_{\eta \wedge \sigma_r})$, a.s.,
and the optional stopping theorem implies that
\[ f(x) = \Exp_x f(X_0) \geq \Exp_x f ( X_{\eta \wedge \sigma_r} ) \geq \Exp_x \left[ f (X_\eta) \1{ \eta < \sigma_r } \right] ,\]
since $f \geq 0$. On $\{ \eta < \infty\}$, we have $f(X_\eta) \geq \inf_{0 \leq x \leq x_0} f(x) \geq \delta$, 
for some $\delta >0$ depending on $x_0$. Hence
$\Pr_x ( \eta < \sigma_r ) \leq \delta^{-1} f(x)$, for every $x \in \ZP$. 
Noting that $\{ \eta < \infty \} = \cup_{r \in \N} \{ \eta < \sigma_r \}$,
continuity along monotone limits shows that, for any $x \in \ZP$,
\[ h (x) = \Pr_x ( \tau < \infty ) \leq \Pr_x ( \eta < \infty ) = \lim_{r \to \infty} \Pr_x ( \eta < \sigma_r )  \leq \delta^{-1} f(x) .\]
This yields the upper bound in~\eqref{eq:h-crude-bound}.

Similarly, take $\nu < 0$ and let $f := f_{\gcr,\nu}$ once more. Now Lemma~\ref{lem:lyapunov-function} shows that $f (X_{n \wedge \eta \wedge \sigma_r} )$, $n \in \ZP$, is a non-negative, bounded submartingale,
and, by optional stopping,
\[ f(x) = \Exp_x f(X_0) \leq \Exp_x f ( X_{\eta \wedge \sigma_r} ) \leq f(r) + \Exp_x \left[ f (X_\eta) \1{ \eta < \sigma_r } \right] ,\]
using the fact that $f$~is non-increasing, so $f(X_{\sigma_r} ) \leq f(r)$, a.s.,
where $\sigma_r$ is defined at~\eqref{eq:sigma-def}. On $\{ \eta < \infty\}$, we have $f(X_\eta) \leq \sup_{0 \leq x \leq x_0} f(x) \leq 1$, say, and so we get
\begin{equation}
    \label{eq:f-lower}
 \Pr_x (\eta < \infty) = \lim_{r \to \infty} \Pr_x ( \eta < \sigma_r ) \geq \lim_{r \to \infty} ( f(x) - f(r) ) = f(x) ,\end{equation}
since $\lim_{r \to \infty} f(r) = 0$. For $n \in \ZP$, define the $\sigma$-algebras  
$\cF_n := \sigma (X_0, X_1, \ldots, X_n)$, and $\cF_\infty := \sigma ( \cup_{n \in \ZP} \cF_n )$; for a stopping time $\kappa$,
let $\cF_\kappa$ denote the $\sigma$-algebra of all $A \in \cF_\infty$ for which $A \cap \{ \kappa \leq n \} \in \cF_n$ for all $n \in \ZP$.
Then $\Pr_x ( \tau < \infty) \geq \Exp_x [ \Pr ( \tau < \infty \mid \cF_{\eta} ) \1{ \eta < \infty } ]$,
and, by irreducibility, $\Pr ( \tau < \infty \mid \cF_{\eta} ) \geq p$, on $\{ \eta < \infty\}$, for some constant $p >0$ (depending on $x_0$). Hence $\Pr_x ( \tau < \infty) \geq p \Pr_x ( \eta < \infty)$, and then~\eqref{eq:f-lower}
yields the lower bound in~\eqref{eq:h-crude-bound}.
 \end{proof}
 
  For $I \subseteq \ZP$, define the first hitting time of set $I$ by $X$ via
 \begin{equation}
     \label{eq:eta-def}
 \eta_I := \inf \{ n \in \ZP : X_n \in I \} . \end{equation}
By irreducibility, $\Pr_x ( \eta_I < \infty ) > 0$
 for every non-empty $I \subseteq \ZP$ and every $x \in \ZP$.
 
 The next result is a key ingredient in our proof. It shows exponential stability of the entrance distribution to an interval, started from a long way above that interval, conditioned on hitting the interval. The proof uses coupling and relies on the uniform irreducibility assumption~\eqref{ass:irreducibility}.
 
 \begin{proposition}
 \label{prop:coupling}
  Suppose that~\eqref{ass:markov} holds for $S= \ZP$, and that~\eqref{ass:bounded-jumps}, \eqref{ass:irreducibility}, and~\eqref{ass:lamperti-transience} hold. 
 For $a \in \ZP$, let $I_a := [ a, a+B] \cap \Z$, where $B$ is the constant in~\eqref{ass:bounded-jumps}.
For all $a \in \ZP$ and all $u \in I_a$, $\theta_a (u) := \lim_{x \to \infty}  \Pr_x ( X_{\eta_{I_a}} = u \mid \eta_{I_a} < \infty )$ exists. Moreover, there exist constants $C< \infty$ and $b >0$ such that, for all $a \in \ZP$ and all $\ell >0$,
 \[ \sup_{x \geq a + B + \ell} \sup_{u \in I_a} \Bigl| \Pr_x ( X_{\eta_{I_a}} = u \mid \eta_{I_a} < \infty ) - \theta_a (u) \Bigr| \leq C \re^{-b \ell} .\]
 \end{proposition}
 
 We will apply Proposition~\ref{prop:coupling} in the following form.
 For $y \in \RP$ let $\lfloor y \rfloor$ denote the largest integer no greater than~$y$.
 
 \begin{corollary}
 \label{cor:coupling}
   Suppose that~\eqref{ass:markov} holds for $S= \ZP$, and that~\eqref{ass:bounded-jumps}, \eqref{ass:irreducibility}, and~\eqref{ass:lamperti-transience} hold. 
   For $A \in \RP$ define $a(x) := x - \lfloor A \log x \rfloor$. Let $\delta >0$. 
   Then there exist $A_\delta, x_B \in \RP$ such that, for all $A \geq A_\delta$, all $x \geq x_B$, and all $|z| \leq B$,
   \begin{equation}
       \label{eq:local-prob}
 \sup_{u \in I_{a(x)}} \Bigl| \Pr_{x+z} ( X_{\eta_{I_{a(x)}}} = u \mid \eta_{I_{a(x)}} < \infty ) - \theta_{a(x)} (u) \Bigr| \leq x^{-\delta}  .   \end{equation} 
  \end{corollary}

 Before proving Proposition~\ref{prop:coupling} and its corollary, 
 we introduce another conditional Markov chain.
For $a, i \in \ZP$,
define
\begin{equation}
    \label{eq:g-def}
    g_a (i) := \Pr_i ( \eta_{I_a} < \infty ) ,
\end{equation}
where $\eta_{I_a}$
is defined at~\eqref{eq:eta-def}. Note that $g_a(i) = 1$ for $i \leq a+B$.
Similarly to $h$ as defined at~\eqref{eq:h-def},
$g_a : \ZP \to (0,1]$ is harmonic for $P$.
Let $q^a_{i,j}$ denote the one-step transition probabilities for $X$ conditioned on $\eta_a < \infty$, i.e.,
the Doob transform relative to~$g_a$: 
\[ q^a_{i,j} := \frac{g_a(j) }{g_a(i)}p_{i,j} , \text{ for all } i, j \in \ZP, i > a + B,\]
and $q^a_{i,j} := p_{i,j}$ for $i \leq a+B$.
Write $Q^a := (q^a_{i,j})_{i,j \in S}$ for the corresponding transition matrix,
and $\bbQ^a_i$ for the law corresponding to initial state~$i \in \ZP$
and transition matrix $Q^a$: i.e., 
for all $n \in \N$ and all $i, x_1, \ldots, x_n \in \ZP$,
 \begin{align*}
\bbQ^a_i ( X_0 =i, X_1 = x_1, \ldots, X_n = x_n )  = q^a_{i,x_1} q^a_{x_1,x_2} \cdots q^a_{x_{n-1},x_n}. \end{align*}
 The following result gives an analogue of Lemma~\ref{lem:doob-h-transform},
 and shows that the uniform irreducibility hypothesis~\eqref{ass:irreducibility} carries over to $\bbQ^a_i$.

 \begin{lemma}
 \label{lem:doob-g-transform}
  Suppose that~\eqref{ass:markov} holds and that $S=\ZP$. Let $a \in \ZP$. 
  Then $\bbQ^a_i (\eta_{I_a} < \infty) = 1$ for all $i \in \ZP$. Moreover, for every $i \in \ZP$, under $\Pr_i$, the law of $(X_0, X_1, \ldots, X_{\eta_{I_a}})$ given $\eta_{I_a} < \infty$
 is the same as the law of $(X_0, X_1, \ldots, X_{\eta_{I_a}})$ under $\bbQ^a_i$. In addition, if~\eqref{ass:bounded-jumps} and~\eqref{ass:irreducibility} hold,
 then there exists $\eps' > 0$ such that, for every $a, i \in \ZP$,
\begin{equation}
    \label{eq:Q-irreducibility}
\max_{1 \leq n \leq m} \bbQ^a_i ( X_n = j ) \geq \eps', \text{ for all } j \in \ZP \text{ with } | j - i | \leq B ,
\end{equation}
where~$m \in \N$ is as in~\eqref{ass:irreducibility}.
  \end{lemma}
 \begin{proof}
 We establish the uniform irreducibility statement for $\bbQ^a$; the other
 statements follow in the same way as the corresponding statements in Lemma~\ref{lem:doob-h-transform}. Let $m \in \N$ and $\eps>0$ be the constants from~\eqref{ass:irreducibility}. Then a consequence of~\eqref{ass:irreducibility} is that, for all $i, j \in \ZP$ with $|j-i| \leq B$,
 there exists $n_{i,j}$ with $n_{i,j} \leq m$ such that, by the Markov property, 
 \[ g_a (i) = \Pr_i ( \eta_{I_a} < \infty ) 
 \geq \Pr_i ( X_{n_{i,j}} = j ) \Pr_j ( \eta_{I_a} < \infty )
 \geq \eps g_a (j) .\]
 Hence $q^a_{i,j} \geq \eps p_{i,j}$ for all $i,j \in\ZP$ with $|i-j| \leq B$.
 It follows that $\bbQ^a_i (X_n = j ) \geq \eps^n \Pr_i (X_n = j)$ for all $i, j, n \in \ZP$, and all $a$. Then~\eqref{eq:Q-irreducibility}
 follows from~\eqref{ass:irreducibility}, with $\eps' = \eps^{m+1}$, with $m \in \N$ as in~\eqref{ass:irreducibility}.
 \end{proof}
 
  \begin{proof}[Proof of Proposition~\ref{prop:coupling}.]
 Fix $a \in \ZP$.
We construct on a single probability space (an embedding of) two coupled copies of $X$ conditioned to reach interval $I_a$, and then stopped, 
one with law $\bbQ_i^a$ and one with law $\bbQ^a_j$.

Set $x_1 := a+B$, and let $\eta := \inf \{ n \in \ZP : X_n \leq x_1 \}$;
by~\eqref{ass:bounded-jumps}, for all $i \in \ZP$ with $i \geq x_1$,
it is the case that $\Pr_i ( \eta = \eta_{I_a} ) = \bbQ^a_i ( \eta = \eta_{I_a} ) = 1$, with $\eta_{I_a}$ as defined at~\eqref{eq:eta-def}.
Moreover, Lemma~\ref{lem:doob-g-transform} says that
$\bbQ^a_i (\eta_{I_a} < \infty) = 1$. Hence $\bbQ^a_i ( \eta < \infty )$ for all $i \in \ZP$.

For initial states $i,j \in \ZP$ with $i, j \geq x_1$, define a Markov chain $(Y,Y')$ on $\ZP^2$ with 
law $\bbQ^a_{i,j}$ such that
$\bbQ^a_{i,j} ( Y_0 = i, \, Y'_0 = j) = 1$, and 
one-step
transition probabilities given by
\begin{itemize}
    \item if $x \vee y \leq x_1$, then $\bbQ^a_{i,j} (Y_{n+1} = x, \, Y'_{n+1} = y \mid Y_n = x, \, Y'_n = y ) =1$; 
    \item if $x > x_1$, then $\bbQ^a_{i,j} ( Y_{n+1} = Y'_{n+1} = z \mid Y_n = Y'_n = x) = q^a_{x,z}$ for all $z$;
    \item if $x > y \vee x_1$, then $\bbQ^a_{i,j} ( Y_{n+1} = z,\, Y'_{n+1} = y \mid Y_n = x,\, Y'_n = y) = q^a_{x,z}$ for all $z$;
    \item if $y > x \vee x_1$, then $\bbQ^a_{i,j} ( Y_{n+1} = x,\, Y'_{n+1} = z \mid Y_n = x,\, Y'_n = y) = q^a_{y,z}$ for all $z$;
\end{itemize}
all other probabilities being zero. 
In words, each coordinate of the process stops as soon as it enters $[0,x_1]$. Otherwise, if $Y$ and $Y'$ coincide, then they jump together according to the transition matrix $Q^a$,
while if they do not coincide, then whichever of $Y, Y'$ is bigger jumps according to $Q^a$, with the other
coordinate remaining fixed.

Define $K := \{ n \in \ZP : Y_n \geq Y'_n \}$ and $K' := \{ n \in \ZP : Y'_n \geq Y_n\}$.
Process $Y$ can jump only at times $K$, process $Y'$ can jump only at times $K'$,
and both processes can jump (together) only at times $K \cap K'$. List the elements of $K$   as a (possibly terminating) sequence as $\nu_0 < \nu_1 < \cdots$ and the elements of $K'$ as $\nu'_0 < \nu'_1 < \cdots$.
Let
 $\rho := \inf \{ n \in \ZP : Y_{\nu_n} \leq x_1\}$,
 and $\rho' := \inf \{ n \in \ZP : Y'_{\nu'_{n}} \leq x_1\}$.

By construction, 
$(Y_{\nu_n}, 0 \leq n \leq \rho)$ under $\bbQ^a_{i,j}$ has the same law as $(X_n, 0 \leq n \leq \eta)$ under $\bbQ^a_i$,  
and 
$(Y'_{\nu'_n}, 0 \leq n \leq \rho')$ under $\bbQ^a_{i,j}$ has the same law as $(X_n, 0 \leq n \leq \eta)$ under $\bbQ^a_j$.
In particular, since $\bbQ^a_i ( \eta < \infty ) = 1$, we have that
$\bbQ^a_{i,j} ( \nu_\rho < \infty ) = \bbQ^a_{i,j} ( \nu'_{\rho'} < \infty ) = 1$ for all $i, j \geq x_1$.
Define $\kappa:= \inf \{ n \in \ZP : Y_n = Y'_n \}$. On the event $\{ \kappa \leq \nu_\rho \vee \nu'_{\rho'} < \infty \}$, we have that $Y_{\nu_\rho} = Y'_{\nu'_{\rho'}}$, and hence
\[ \sup_{u \in I_a} \left| \bbQ^a_{i,j} ( Y_{\nu_\rho} = u ) - \bbQ^a_{i,j} ( Y'_{\nu'_{\rho'}} = u) \right|
\leq \bbQ^a_{i,j} ( Y_{\nu_\rho} \neq Y'_{\nu'_{\rho'}} ) \leq 
1 - \bbQ^a_{i,j} ( \kappa \leq \nu_\rho \vee \nu'_{\rho'} < \infty ) .\]
It follows that 
\begin{align}
    \label{eq:coupling-bound}
& \sup_{u \in I_a} \bigl| \Pr_i ( X_\eta = u \mid \eta < \infty ) - \Pr_j ( X_\eta = u \mid \eta < \infty ) \bigr| \nonumber\\
& {} \qquad\qquad\qquad {} \leq \bbQ^a_{i,j} ( \kappa > \nu_\rho \vee \nu'_{\rho'} )
+  \bbQ^a_{i,j} (  \nu_\rho \vee \nu'_{\rho'} = \infty ) \nonumber\\
& {} \qquad\qquad\qquad {} = \bbQ^a_{i,j} ( \kappa > \nu_\rho \vee \nu'_{\rho'} ),
\end{align}
since $\bbQ^a_{i,j} ( \nu_\rho < \infty ) = \bbQ^a_{i,j} ( \nu'_{\rho'} < \infty ) = 1$. 
 It remains to bound the right-hand side of~\eqref{eq:coupling-bound}. The idea is that at each time~$n$ for which $Y_n', Y_n$
 are such that $|Y_n' - Y_n | \leq B$, the uniform irreducibility bound~\eqref{eq:Q-irreducibility} from Lemma~\ref{lem:doob-g-transform} shows that there is uniformly positive probability of coupling within $m$ steps; if not, one can try again later, and the total number of coupling attempts before reaching $I_a$ will grow linearly in $i \wedge j$, by~\eqref{ass:bounded-jumps}. We give the details.
 
 Take $i,j > x_1 + (k+1) m B$ for some $k \in \N$. If $|i-j | >B$, then the
 maximal of the two processes will move, and, since $\bbQ^a_i ( \eta < \infty) = \bbQ^a_j (\eta < \infty) =1$ (by Lemma~\ref{lem:doob-g-transform}), we will eventually have $| Y_n - Y'_n | < B$ while $Y_n, Y_n' \geq x_1 + k m B$. Hence we may, without loss of generality, suppose that $| i - j | \leq B$. It then follows from~\eqref{eq:Q-irreducibility} that for some $n \leq m$ and some $\eps''>0$, $\bbQ^a_{i,j} ( Y_n = Y'_n ) \geq \eps''$, in which case $\kappa < \nu_\rho \vee \nu'_{\rho'}$. Otherwise, we can apply the same argument at time $m$, by which time $Y_n, Y_n' \geq x_1 + (k-1) m B$. 
 Given $i, j > x_1$, choose the constant~$k \in \ZP$ so that $i \wedge j \geq x_1 + k m B$, i.e., $k = \lfloor \frac{(i \wedge j) - x_1}{mB} \rfloor$. Then iterating the above coupling argument~$k$ times, we get
 \[ \bbQ^a_{i,j} ( \kappa > \nu_\rho \vee \nu'_{\rho'} ) \leq (1-\eps'')^k \leq C \re^{-b [ (i \wedge j ) - a]} , \text{ for all } i, j \geq a+B, \]
 where the constants $C< \infty$ and $b>0$ depend on $B, m$, and $\eps''$, but do not depend on $a$, $i$, or $j$. Then~\eqref{eq:coupling-bound} yields,
 for all  $i, j \geq a+B$,
 \begin{equation}
    \label{eq:exponential-bound-cauchy}
\sup_{u \in I_a} \bigl| \Pr_i ( X_\eta = u \mid \eta < \infty ) - \Pr_j ( X_\eta = u \mid \eta < \infty ) \bigr| \leq C \re^{-b [ (i \wedge j ) - a]} .
\end{equation} 
 In particular, the bound~\eqref{eq:exponential-bound-cauchy} shows that for each $u \in I_a$,
 $\Pr_i ( X_\eta = u \mid \eta < \infty )$ is a Cauchy sequence in~$i$, and hence $\theta_a (u) = \lim_{i \to \infty} \Pr_i ( X_\eta = u \mid \eta < \infty )$ exists; the exponential convergence rate also follows from~\eqref{eq:exponential-bound-cauchy}.
  \end{proof}
  
  \begin{proof}[Proof of Corollary~\ref{cor:coupling}.]
  Take $a(x) := x - \lfloor A \log x \rfloor$,
  with $A \geq 1$ and $x \geq \re^{1+2B}$, so that $\lfloor A \log x\rfloor - 2B > 0$.
    Then for $|z| \leq B$,
    $x +z \geq  a(x) + B + \lfloor A \log x\rfloor - 2B$. 
  We can then apply Proposition~\ref{prop:coupling} with $a = a(x) \in \ZP$
  and $\ell = \lfloor A \log x\rfloor - 2B >0 $ to obtain
  \[  \sup_{u \in I_{a(x)}} \Bigl| \Pr_{x+z} ( X_{\eta_{I_{a(x)}}} = u \mid \eta_{I_a} < \infty ) - \theta_{a(x)} (u) \Bigr| \leq C \re^{2B b} \re^{-b \lfloor A \log x\rfloor} ,\]
  for all $|z| \leq B$ and all $x \geq \re^{1+2B} =: x_B$. 
Hence we may choose $A \geq A_\delta$ large enough so that~\eqref{eq:local-prob} holds, as claimed.
  \end{proof}
 
 The next result enables us to express the ratio of return probabilities to hitting probabilities of a relatively nearby interval that is nevertheless far from the origin, and hence more amenable to estimation by Lyapunov functions based on our asymptotic assumptions.
 
 \begin{lemma}
 \label{lem:h-ratio-bound}
  Suppose that~\eqref{ass:markov} holds for $S= \ZP$, and that~\eqref{ass:bounded-jumps}, \eqref{ass:irreducibility}, and~\eqref{ass:lamperti-transience} hold. 
  Let  $a(x) = x - \lfloor A \log x \rfloor$,
 for $A \geq A_{\gcr+3}$ the constant in Corollary~\ref{cor:coupling}.
Then
 \[ \sup_{|z| \leq B} \left| \frac{h(x+z)}{h(x)} - \frac{\Pr_{x+z} ( \eta_{I_{a(x)}} < \infty ) }{\Pr_{x} ( \eta_{I_{a(x)}} < \infty )} \right| = O ( x^{-2} ), \text{ as } x \to \infty .\]
  \end{lemma}
  
  Before proving Lemma~\ref{lem:h-ratio-bound}, we 
  examine the probabilities $\Pr_{x+z} ( \eta_{I_{a(x)}} < \infty )$. 
    Take $a(x) = x - \lfloor A \log x \rfloor$, as in Corollary~\ref{cor:coupling}.
 For $\nu \in \R$, $x \in \RP$, and $z \in \R$ with $x \geq 1$ and $x+z \geq 0$, define
  \begin{equation}
      \label{eq:R-def}
 R_\nu (x,z) := \frac{f_{\gcr,\nu} (x+z)}{\sum_{u \in I_{a(x)}} \theta_{a(x)} (u) f_{\gcr,\nu} (u)} ,
   \end{equation}
   where $\theta_a(u)$ is as defined in Proposition~\ref{prop:coupling},
   $f$ is as defined at~\eqref{eq:f-def}, and~$\gcr$ is given by~\eqref{eq:gamma-c}. Note that, for every $\gamma >0$, for all $b : \ZP \to \RP$ with $\lim_{x \to \infty} b(x)/x = 0$,
   \begin{equation}
       \label{eq:f-smooth}
       \lim_{x \to \infty} \sup_{|z| \leq b(x)} \left| \frac{f_{\gamma,\nu} (x+z)}{f_{\gamma,\nu} (x)} - 1 \right| = 0 .
   \end{equation}
   Since $\sum_{u \in I_a} \theta_a (u) = 1$, it follows from~\eqref{eq:R-def} and~\eqref{eq:f-smooth} that
   \begin{equation}
       \label{eq:R-limit}
       \lim_{x \to \infty} \sup_{|z| \leq B} \left| R_\nu (x,z) - 1 \right| = 0.
   \end{equation}
   The following lemma combines the Lyapunov function ideas of Lemma~\ref{lem:crude-bound} with the stability of the interval entrance distribution from Corollary~\ref{cor:coupling} to obtain refined hitting probability bounds.
 
 \begin{lemma}
 \label{lem:eta-R-bound}
  Suppose that~\eqref{ass:markov} holds for $S= \ZP$, and that~\eqref{ass:bounded-jumps}, \eqref{ass:irreducibility}, and~\eqref{ass:lamperti-transience} hold. 
 Let $\eps >0$ and  $a(x) = x - \lfloor A \log x \rfloor$,
 for $A \geq A_{\gcr+3}$ the constant in Corollary~\ref{cor:coupling}. Then for all $x \in \N$ and all $|z| \leq B$, as $x \to \infty$,
 with $R$ as defined at~\eqref{eq:R-def}, 
 \[ R_{-\eps} (x,z)  + O ( x^{-2} ) 
 \leq
 \Pr_{x+z}  ( \eta_{I_{a(x)}} < \infty) \leq R_{\eps} (x,z)  + O ( x^{-2} ). \]
 \end{lemma}
  \begin{proof}
For $r \in \RP$, define $\sigma_r$ by~\eqref{eq:sigma-def}, 
and, for ease of notation, write $\eta := \eta_{I_{a(x)}}$. 
Fix $\eps >0$.
If follows from Lemma~\ref{lem:lyapunov-function} that, 
for all $x$ sufficiently large, 
$f_{\gcr,\eps} (X_{n \wedge \eta  \wedge \sigma_r} )$, $n \in \ZP$,
is a non-negative, uniformly bounded supermartingale with
$\lim_{n \to \infty} f_{\gcr,\eps} (X_{n \wedge \eta  \wedge \sigma_r} ) = f_{\gcr,\eps} (X_{\eta \wedge \sigma_r})$, a.s. Hence, by optional stopping,
\[ f_{\gcr,\eps} (x+z) \geq \Exp_{x+z} f_{\gcr,\eps} \bigl(X_{\eta \wedge \sigma_r} \bigr)
\geq \Exp_{x+z} \left[ f_{\gcr,\eps} \bigl(X_{\eta} \bigr) \1 { \eta < \sigma_r } \right] .\]
Thus, by monotone convergence,
\[  f_{\gcr,\eps} (x+z) \geq \lim_{r \to \infty} \Exp_{x+z} \left[f_{\gcr,\eps} \bigl(X_{\eta} \bigr) \1 { \eta < \sigma_r } \right]
= \Exp_{x+z} \left[ f_{\gcr,\eps} \bigl(X_{\eta} \bigr) \1 { \eta < \infty } \right] .\]
For $|z| \leq B$ and $x+z \in \ZP$, we get from an application of~\eqref{eq:local-prob} with $\delta = \gcr +3$,
provided $A \geq A_{\gcr+3}$ the constant in Corollary~\ref{cor:coupling},
\begin{align*}
     f_{\gcr,\eps} (x+z) & \geq 
 \Pr_{x+z} \bigl(  \eta < \infty \bigr)  \sum_{u \in I_{a(x)}} f_{\gcr,\eps} (u) \Pr_{x+z} ( X_{\eta} = u \mid \eta < \infty ) \\
& = \Pr_{x+z} \bigl(  \eta < \infty \bigr) \left[ \sum_{u \in I_{a(x)}} f_{\gcr,\eps} (u)
 \theta_{a(x)} (u) + O ( x^{-\gcr-3} ) \right] ,\end{align*}
 using boundedness of $f_{\gcr,\eps}$
 and the fact that $I_{a(x)}$ has $B+1$ elements. With~\eqref{eq:R-limit} and the fact that $x^{-\gcr-3} / f_{\gcr,\eps} (x) = O(x^{-2})$, this gives the upper bound in the lemma. 

On the other hand, Lemma~\ref{lem:lyapunov-function} shows that, for all $x$ sufficiently large, the process 
$f_{\gcr,-\eps} (X_{n \wedge \eta  \wedge \sigma_r} )$, $n \in \ZP$,
is a non-negative, uniformly bounded submartingale with
$\lim_{n \to \infty} f_{\gcr,-\eps} (X_{n \wedge \eta  \wedge \sigma_r} ) = f_{\gcr,-\eps} (X_{\eta \wedge \sigma_r})$, a.s. Hence, by optional stopping,
\begin{align*} f_{\gcr,-\eps} (x+z) & \leq \Exp_{x+z} f_{\gcr,-\eps} (X_{\eta \wedge \sigma_r} )
\\
& \leq \Exp_{x+z}   \left[ f_{\gcr,-\eps} \bigl(X_{\eta} \bigr)  \1 { \eta < \sigma_r } \right] 
+ f_{\gcr,-\eps} (r ) \Pr_{x+z} \bigl( \sigma_r <  \eta \bigr) 
.\end{align*}
Hence, taking $r \to \infty$, since $f_{\gcr,-\eps} (r) \to 0$, we get
\[ f_{\gcr,-\eps} (x+z) \leq \Exp_{x+z} \left[ f_{\gcr,-\eps}  \bigl(X_{\eta} \bigr)  \1 { \eta < \infty } \right] .\]
Another application of~\eqref{eq:local-prob} now yields the lower bound in the lemma.
 \end{proof}
  
  Now we can complete the proof of Lemma~\ref{lem:h-ratio-bound}.
  
  \begin{proof}[Proof of Lemma~\ref{lem:h-ratio-bound}.]
  First observe that, by~\eqref{ass:bounded-jumps}, whenever $x \in \ZP$ has $x > a + B$,
  \begin{align}
  \label{eq:h-strong-markov}
     h(x) = \Pr_x ( \tau < \infty ) & = \Pr_x (\tau < \infty, \, \eta_{I_a} < \infty) \nonumber\\
     & = \Exp_x \left[ \Pr ( \tau < \infty \mid \cF_{\eta_{I_a}} ) \1 { \eta_{I_a} < \infty} \right] \nonumber\\
     & = \Exp_x \left[ h ( X_{\eta_{I_a}} ) \1 { \eta_{I_a} < \infty} \right] ,
  \end{align}
  by~\eqref{eq:h-def} and the strong Markov property applied at time $\eta_{I_a}$. 
   Combining
the $\delta = \gcr +3$ case of~\eqref{eq:local-prob} with~\eqref{eq:h-strong-markov}, applied at $x+z \in \ZP$ and with $a = a(x)$
for $A \geq A_{\gcr+3}$ the constant in Corollary~\ref{cor:coupling},
we obtain
  \begin{equation}
      \label{eq:h-with-error}
 \sup_{|z| \leq B}\, \biggl| h(x+z) - \Pr_{x+z}  ( \eta_{I_{a(x)}} < \infty) \sum_{u \in I_{a(x)}} h(u) \theta_{a(x)} (u) \biggr| = O (x^{-\gcr-3} ) , \text{ as } x \to \infty.  \end{equation}
  Lemma~\ref{lem:eta-R-bound}, with the observation~\eqref{eq:R-limit},
shows that 
\begin{equation}
\label{eq:x-return}
\lim_{x \to \infty} \Pr_{x+z}  ( \eta_{I_{a(x)}} < \infty) = 1, \text{ for all } z, \end{equation}
   while, by  Lemma~\ref{lem:crude-bound}, for any $\eps >0$,    for all $x$ sufficiently large,
   \[ \sum_{u \in I_{a(x)}} h(u) \theta_{a(x)} (u) \geq \inf_{u \in I_{a(x)}} h(u) \geq x^{-\gcr} \log^{-\eps} x. \]
 Thus~\eqref{eq:h-with-error}
   implies that $h(x+z) = ( \Pr_{x+z}  ( \eta_{I_{a(x)}} < \infty) + O (x^{-2} ) ) \sum_{u \in I_{a(x)}} h(u) \theta_{a(x)} (u)$,
    uniformly for $|z| \leq B$,
   and then using~\eqref{eq:x-return} the 
   claimed result follows.
  \end{proof}
  
 \begin{proof}[Proof of Theorem~\ref{thm:h-estimate}.]
 Let $\eps >0$. 
 From Lemma~\ref{lem:h-ratio-bound} together with Lemma~\ref{lem:eta-R-bound}
 and~\eqref{eq:R-limit}, we have that, uniformly in $|z| \leq B$, as $x \to \infty$,
 \begin{equation}
     \label{eq:h-R-bounds}
 \frac{R_{-\eps} (x,z)}{R_\eps (x,0)} + O (x^{-2} ) \leq \frac{h(x+z)}{h(x)} \leq \frac{R_\eps (x,z)}{R_{-\eps} (x,0)} + O ( x^{-2} ) . \end{equation}
Fix $A \geq A_{\gcr+3}$, where $A_{\gcr+3}$ is the constant in Corollary~\ref{cor:coupling}, and take $x_A \geq 1$ for which
$a(x) = x - \lfloor A \log x \rfloor$ satisfies $a(x) \geq 1$ for all $x \geq x_A$. 
For $\nu \in \R$ and $x \geq x_A$, define
 \[ \Theta_\nu (x) := \sum_{u \in I_{a(x)}} \theta_{a(x)}(u) f_{\gcr,\nu} (u)  .\]
 For a given $A \in \RP$ and $\delta >0$, an application of the mean value theorem shows that we can choose $\eps >0$ and $x_A' \geq x_A$
 such that, for all $| \nu| \leq \eps$,
 \begin{equation}
     \label{eq:log-bounds}
 \sup_{y \in I_{a(x)}} \left| \frac{\log^\nu y}{\log^\nu x} - 1 \right| \leq \frac{\delta}{x}, \text{ for all } x \geq x_A' . \end{equation}
 Since
 \[ \left| \frac{\Theta_\nu (x) - \Theta_0(x) \log^\nu x}{\Theta_0(x)  \log^\nu x} \right| \leq \frac{\sum_{u \in I_{a(x)}} \theta_{a(x)}(u) u^\gcr \left| \log^\nu u - \log^\nu x\right| }{\sum_{u \in I_{a(x)}} \theta_{a(x)}(u) u^\gcr \log^\nu x} 
 \leq   \sup_{u \in I_{a(x)}} \left| \frac{\log^\nu u}{\log^\nu x} - 1 \right|, 
 \]
 it follows from~\eqref{eq:log-bounds} that, for $\eps >0$ sufficiently small, for all $|\nu| \leq \eps$
 and all $x \geq x_A'$, 
 \begin{align*}
      \left| \frac{\Theta_\nu (x)}{\Theta_0(x) \log^\nu x} - 1 \right| & \leq \frac{\delta}{x} ,
 \text{ and }
       \left| \frac{\Theta_{\nu} (x)}{\Theta_{-\nu}(x)} \log^{-2\nu} x - 1 \right|  \leq \frac{\delta}{x}. 
 \end{align*}
 Since $R_\nu$ as defined at~\eqref{eq:R-def} satisfies $R_\nu (x,z) = f_{\gcr,\nu} (x+z) / \Theta_\nu (x)$, we get
 \begin{align}
 \label{eq:R-ratio}
& {} \quad {} \sup_{|z| \leq B} \left| \frac{R_\nu (x,z)}{R_{-\nu} (x,0)}
 -  \frac{ f_{\gcr,\nu} (x+z)}{f_{\gcr,-\nu} (x) } \log^{-2\nu} (x) \right| \nonumber\\
& \leq 
  \left| \frac{\Theta_{-\nu} (x)}{\Theta_\nu (x)} \log^{2\nu} x  - 1 \right|
   ( \log^{-2\nu} x ) \sup_{|z| \leq B} \frac{ f_{\gcr,\nu} (x+z)}{f_{\gcr,-\nu} (x) }
  \nonumber\\
& \leq   \frac{\delta}{x}  ( \log^{-2\nu} x ) \sup_{|z| \leq B} \frac{ f_{\gcr,\nu} (x+z)}{f_{\gcr,-\nu} (x) } ,
 \end{align}
for all $x \geq x_A'$. Here, for $|z| \leq B$,
\begin{align} 
\label{eq:f-ratio}
\frac{ f_{\gcr,\nu} (x+z)}{f_{\gcr,-\nu} (x) } \log^{-2\nu} x & = \left( 1 + \frac{z}{x} \right)^{-\gcr} \left( \frac{\log (x+z)}{\log x } \right)^\nu \nonumber\\
& = \left(1 - \frac{\gcr z}{x} + O (x^{-2} ) \right)\left( 1 + O ( x^{-1} \log^{-1} x ) \right) \nonumber\\
& = 1 - \frac{\gcr z}{x} + O (x^{-1} \log^{-1} x ) .
\end{align}
Thus from~\eqref{eq:R-ratio} and~\eqref{eq:f-ratio} we obtain, for all $| \nu | \leq \eps$ and all $x$ sufficiently large, 
\[ \sup_{|z| \leq B} \left| \frac{R_\nu (x,z)}{R_{-\nu} (x, 0)} - 1 + \frac{\gcr z}{x} \right| \leq \frac{\delta}{x} .\]
Using this in~\eqref{eq:h-R-bounds}, we obtain, for any $\delta>0$, for all $x$ sufficiently large
\[  \sup_{|z| \leq B} \left| \frac{h(x,z)}{h(x)} - 1 + \frac{\gcr z}{x} \right| \leq \frac{\delta}{x} .\]
Since $\delta>0$ was arbitrary, the result follows.
 \end{proof}
 
 Finally, we can complete the proof of Theorem~\ref{thm:conditional-moments}.
 
 \begin{proof}[Proof of Theorem~\ref{thm:conditional-moments}.]
 We have from Theorem~\ref{thm:h-estimate}
 that, as $x \to \infty$, 
 \[ \frac{h(x+z)}{h(x)} = 1 - \frac{\gcr z}{x} + o (x^{-1} ) ,\]
 uniformly for $|z| \leq B$. Then from the $k=1$ case of~\eqref{eq:tmu-k-def} and~\eqref{ass:lamperti-transience}, we get
 \[ \tmu_1 (x) = \mu_1 (x) - \frac{\gcr}{x} \mu_2 (x) + o (x^{-1}) = \frac{c - \gcr s^2 + o(1)}{x} , \text{ as } x \to \infty.\]
 Since, by~\eqref{eq:gamma-c}, $\gcr s^2 = 2c-s^2$, this gives the result for $\tmu_1$.
 A similar argument, based on the $k=2$ case of~\eqref{eq:tmu-k-def}, shows that $\tmu_2 (x) = \mu_2 (x) + O (1/x)$, as $x\to\infty$.
 \end{proof}

 \section{Discussion}
 \label{sec:discussion}

 \subsection{Multidimensional random walks}
 \label{sec:random-walks}
 
 In this section, we draw attention to a relevant comparison between our main result and what is known about strong transience
for multidimensional random walks. 
 Let $Z, Z_1, Z_2, \ldots$ be a sequence of i.i.d.~random variables in~$\Z^d$, $d \in \N$.
Let $S = (S_0, S_1, \ldots)$ be the associated random walk, given by $S_n := \sum_{i=1}^n Z_i$, $n \in \ZP$. 
Suppose that $S$ is genuinely $d$-dimensional, i.e., $\supp Z$ is not contained in any $(d-1)$-dimensional subspace
of $\R^d$. 
Denote by $\phi (u) := \Exp \re^{iu^\tra Z}$, $u \in \R^d$, the characteristic function of $Z$.

The following classification criterion, 
of Chung--Fuchs type, 
for strong transience is contained in Theorem~4.3 of~\cite{port66}
and Theorem~5 of~\cite{takeuchi}.

\begin{proposition}[Port, 1966; Takeuchi, 1967]
\label{prop:random-walk}
Let $\beta \in (0,\infty)$. 
The random walk~$S$ is $\beta$-strong transient if and only if
\[ \lim_{ t \uparrow 1} \int_{[-\pi,\pi]^d} \Re \left\{ \frac{1}{(1-t\phi (u))^{1+\beta} } \right\} \ud u  < \infty .\]
\end{proposition}

Versions of Proposition~\ref{prop:random-walk} for L\'evy processes can be found in~\cite{sato,sw}; further extensions include those in~\cite{sw2,sandric}.

The following result is due to Port (Theorem~4.4 of~\cite{port66}, for the case $\beta=1$; see also~\cite[p.~144]{hughes}) and Takeuchi (Theorem~6 of~\cite{takeuchi});
a L\'evy process analogue can be found in~\cite{sw},
while for the case of Brownian motion in $\R^3$, Spitzer~\cite{spitzer64} attributes the result to Joffe~\cite{joffe}.

\begin{proposition}
\label{prop:high-dimensions}
If $d > 2 \beta +2$, then $S$ is $\beta$-strong transient.
On the other hand, if $\Exp [ \|Z\|^2 ] < \infty$, $\Exp Z = 0$,
and $d \leq 2 \beta +2$, then $S$ is not $\beta$-strong transient.
\end{proposition}

For the proof of this result, we recall some terminology about lattice random walks,
and a consequence of the multidimensional lattice local limit theorem: 
for reference see \cite[Ch.~5]{BR}, \cite[\S 7]{spitzerbook}, or~\cite[\S A]{lw}.
If $Z$ generates a genuinely $d$-dimensional, lattice random walk~$S$,
Lemma~21.4 of~\cite{BR} shows that there is a unique minimal subgroup $L$ of $\R^d$ such that
$\Pr ( Z \in b + L ) =1$ for any $b \in \R^d$ with $\Pr ( Z = b ) >0$.
The subgroup $L$ is of the form $L = H \Z^d$ for a non-singular, $d$-dimensional matrix $H$,
and minimality of $L$ is equivalent to a condition on $\det H$,
as well as the condition that $| \varphi ( u ) | = 1$ if and only if
$u \in 2 \pi ( H^\tra)^{-1} \Z^d$
(see e.g.~Lemma~A.4 of~\cite{lw}).

The \emph{period} of $S$ is the maximal $\ell \in \N$ such that $\Pr ( S_{\ell n} = 0) >0$ for all $n \in \N$.
If $\ell =1$, we say the random walk $S$ is \emph{aperiodic}. 
If $X$ generates a walk $S$ of period $\ell \geq 2$,
then the increment $\tZ := Z_1 + \cdots + Z_{\ell}$ generates an aperiodic random walk $\tS$, and then
$ \Pr ( S_{\ell n} = 0 ) = \Pr ( \tS_n = 0 )$. Moreover, since $\Pr ( \tZ = 0 ) >0$,
the increment $\tZ$ has $\Pr ( \tZ \in \widetilde{L} ) =1$ for an associated minimal lattice $\widetilde{L}$, with no shift.
Thus it suffices to assume that our original $Z$ is such that $b=0$. With the  transformation $\tZ = H^{-1} Z$, 
we can further reduce to the case where $H = I$ (the identity). This is the setting which Spitzer calls ``strong aperiodicity''~\cite[pp.~42, 75]{spitzerbook}.

Suppose that $\Exp [ \| Z \|^2 ] < \infty$ and $\Exp Z = 0$. Then (since $Z$ is genuinely $d$-dimensional)
there is a positive definite, symmetric, $d$-dimensional matrix $\Sigma$ such that
$\Exp [ Z Z^\tra ] = \Sigma$. Any non-singular linear transformation (such as through the reduction $\tZ = H^{-1} Z$) merely transforms the covariance matrix $\Sigma$. Keeping track of the various reductions,
the multidimensional lattice local central limit theorem~\cite[pp.~75--77]{spitzerbook}
implies the following.

\begin{lemma}
    \label{lem:llt}
    Suppose that $Z$ has a lattice distribution, generates a genuinely $d$-dimensional
    random walk $S$ with period $\ell \in \N$, and satisfies $\Exp [ \| Z \|^2 ] < \infty$ and $\Exp Z = 0$.
    Then there is a constant $\rho >0$ depending only on $d$ and $\Sigma := \Exp [ Z Z^\tra ]$ such that, for all $n \in \N$, $\Pr ( S_{\ell n} = 0) \geq \rho n^{-d/2}$.
\end{lemma}

\begin{proof}[Proof of Proposition~\ref{prop:high-dimensions}]
For a genuinely $d$-dimensional random walk, one has the estimate 
$\sup_{y \in \Z^d} \Pr ( S_n = y ) \leq C n^{-d/2}$ for some $C < \infty$ and all $n \in \N$:
this follows from concentration function estimates~\cite[Thm.~6.2]{esseen} or~\cite[p.~72]{spitzerbook}. 
Hence
\[ \sup_{y \in \Z^d} U_\beta (0,y) \leq C \sum_{n \in \N} n^{\beta - (d/2)} ,\]
and then Lemma~\ref{lem:equivalence} yields $\beta$-strong transience for  $d > 2 \beta +2$. On the other hand, if $\Exp [ \|Z\|^2 ] < \infty$ and $\Exp Z = 0$, then the multidimensional lattice local central limit theorem (Lemma~\ref{lem:llt}) and another appeal to Lemma~\ref{lem:equivalence} complete the proof.
\end{proof}

We indicate an analogy between the above results and our result for Lamperti
processes, Theorem~\ref{thm:lamp-st}. 
Consider the case where $S$ is symmetric simple random walk on~$\Z^d$, $d \in \N$;
by Proposition~\ref{prop:high-dimensions}, $S$ is 
$\beta$-strong transient for $d > 2\beta +2$ and not if $d \leq 2 \beta +2$.
Define $X_n := \| S_n \|$, where $\| \, \cdot \, \|$
is the Euclidean norm on $\R^d$. Then $X$ is not a Markov process, but nevertheless satisfies Lamperti-type increment moment conditions. Indeed,
a calculation (see e.g.~\cite[\S 1.3]{mpw}) shows that
\begin{align*}
     \Exp [ X_{n+1} - X_n \mid S_n = z ] & = \left( \frac{d-1}{2d} \right) \frac{1}{\| z\|} + O ( \|z\|^{-2} );\\
         \Exp [ (X_{n+1} - X_n)^2 \mid S_n = z ] & =  \frac{1}{d} + O ( \|z\|^{-1} ),
     \end{align*}
     where the error terms are uniform in $z \in \Z^d$ as $\| z\| \to \infty$. 
     Thus a non-Markovian version of the Lamperti drift conditions~\eqref{ass:lamperti-transience} holds, with $c = (d-1)/(2d)$ and $s^2 = 1/d$. 
     Lamperti's recurrence theory (see~\cite{lamp1} and~\cite[Ch.~3]{mpw}) extends to this non-Markovian setting, and again shows that we have transience when $2c > s^2$, i.e., $d > 2$. This argument establishes P\'olya's recurrence theorem and extends to any case where $\Exp [ \| Z \|^2 ] < \infty$ and $\Exp Z = 0$. Moreover, we see that
     our Theorem~\ref{thm:lamp-st} is consistent with Proposition~\ref{prop:high-dimensions}, even though our theorem does not apply directly, since $X = \| S\|$ is not Markov. The main obstacle to extending the method of the present paper to the multidimensional  setting seems to be Proposition~\ref{prop:coupling}, which relies on the one-dimensional nature of the problem.

\subsection{Critical branching processes with migration}
\label{sec:branching}

Let $\Xi := ( \xi_{n,i} ; n \in \N, i \in \N)$ be an array of independent~$\ZP$-valued random variables
with distribution identical to that of a random variable $\xi$,
and let $\zeta, \zeta_1, \zeta_2, \ldots$ be i.i.d.~$\Z$-valued random variables
independent of the collection $\Xi$. Define $W_0 := w_0 \in \N$ (the initial
population size) and for $n \in \N$ set
\[ W_{n} :=  \max \left\{ \sum_{i=1}^{W_{n-1}} \xi_{n,i} +  \zeta_{n} , 0 \right\}  . \]
The nonnegative Markov chain
$(W_n ; n \in \ZP)$
is a \emph{branching process with migration} with \emph{offspring distribution} $\xi$
and \emph{migrant distribution} $\zeta$.
 When $\zeta >0$ this represents immigration
from outside the population, and $\zeta <0$ represents emigration; the population size
$W_n$ at generation $n$ cannot go~negative. 
We denote by $\taue := \inf \{ n \in \ZP : W_n = 0 \}$; the event $\taue < \infty$
 corresponds to \emph{extinction} of the population. 
Note, however, that if $\Pr (\zeta > 0) > 0$ then immigration will eventually restart the process, so $0$ is not necessarily an absorbing state.
 A transformation of $W_n$ yields a Lamperti process
 in the following sense.

\begin{proposition}
\label{prop:branching-lamperti}
Suppose that 
there exists $p >2$
such that $\Exp [ |\xi|^p ] < \infty$ and
$\Exp [ | \zeta |^p ] < \infty$. Assume that
$\Exp \xi = 1$, $\Var \xi = \sigma^2 \in (0,\infty)$,
    and $\Exp \zeta = \theta \in \R$. 
    Then the process $X_n := \sqrt{ W^+_n}$ is a Markov chain
    on the countable state space $\bbX:=\sqrt{\ZP}$, for which
    \begin{equation}
        \label{eq:bp-moments}
     \sup_{x \in \bbX} \Exp [ | X_{n+1} - X_n |^p \mid X_n = x ] < \infty,
\end{equation}
        and $\mu_1, \mu_2 : \bbX \to \R$ defined analogously to~\eqref{eq:markov-law} satisfying, for an~$\eps>0$ depending on~$p$, 
    \[ \mu_1 (x) = \frac{4\theta-\sigma^2}{8x} + O (x^{-1 -\eps}), \text{ and }
    \mu_2 (x) = \frac{\sigma^2}{4} + O ( x^{-\eps}) , \text{ as } x \to \infty. \]
        \end{proposition}

From Proposition~\ref{prop:branching-lamperti}, it is a standard application of Lamperti's recurrence classification~\cite[Thm.~3.1, p.~320]{lamp1} to obtain the following result, which is mostly contained in 
Theorem~1 of Pakes~\cite{pakes}
(see e.g.~\cite[pp.~114--115]{mpw} for a similar application in another branching-process context).

    \begin{corollary}
        Under the conditions of Proposition~\ref{prop:branching-lamperti},
        it holds that if $2\theta > \sigma^2$, one has $\Pr ( \taue = \infty ) >0$ (i.e., survival is possible),
    while if  $2\theta \leq \sigma^2$ one has $\Pr ( \taue < \infty ) =1$ (extinction is certain); moreover, if $\theta < 0$ then $\Exp \taue < \infty$ while if $\theta \geq 0$ then $\Exp \taue = \infty$.
    \end{corollary}

The process $X_n$ from Proposition~\ref{prop:branching-lamperti}
is of Lamperti type (the fact that the state-space is $\bbX$ and not $\ZP$ is unimportant),
but it does not satisfy the bounded jumps hypothesis~\eqref{ass:bounded-jumps},
so our Theorem~\ref{thm:lamp-st} on strong transience does not apply.
One might reasonably expect, however, the conclusions of Theorem~\ref{thm:lamp-st} 
would still apply, at least assuming some sufficiently strong moments conditions on $\xi$ and $\zeta$;
some positive evidence in this direction is provided by work of Kosygina \& Zerner~\cite{kz} (see Remark~\ref{rem:branching}). We formulate the following problem to address the general case.

\begin{problem}
     \label{problem:branching}
    Under the conditions of Proposition~\ref{prop:branching-lamperti},
    find effective conditions which ensure that if $2 \theta   > (\beta +1) \sigma^2$
    it holds that $\Exp [ \taue^\beta \mid \taue < \infty] < \infty$,
     while if $2 \theta   < (\beta +1) \sigma^2$
     it holds that $\Exp [ \taue^\beta \mid \taue < \infty] = \infty$.
 \end{problem}

\begin{remark}
    \label{rem:branching}
  The example where $\Pr ( \xi = k ) = 2^{-1-k}$, $k \in \ZP$ (a shifted geometric distribution)  
  has $\Exp \xi = 1$ and $\Var \xi = 2 = \sigma^2$, and so in this case
  the result proposed in Problem~\ref{problem:branching} would say that $\Exp [ \tau \mid \tau < \infty ] < \infty$
  if $\theta > 2$, but $\Exp [ \tau \mid \tau < \infty ] =\infty$
  if $\theta < 2$. For the shifted-geometric example (in fact, a somewhat more general class of examples, allowing
  some partial non-independence), this putative result is indeed true, as has been established by Kosygina \& Zerner~\cite{kz}.
\end{remark}

\begin{proof}[Proof of Proposition~\ref{prop:branching-lamperti}]
Since $x \mapsto \sqrt{x}$ is a bijection over $\RP$, $X_n = \sqrt{ W_n}$ is Markov. To study the increments of $X_n$, we first study the increments of $W_n$. 
Write  
\[ \Delta = \sum_{i=1}^{W_0} ( \xi_{1,i} -1) + \zeta_1 ; \]
then $W_1 - W_0 = \max \{ \Delta , - W_0 \}$.
In particular, since $\Exp \xi = 1$, for $w \in \ZP$,
\begin{equation}
    \label{eq:bp-drift}
 \Exp [ \Delta \mid W_0 = w ] = \Exp  \zeta  = \theta .\end{equation}
By independence of the $\xi_{1,i}$,
$\Var \sum_{i=1}^w   \xi_{1,i} = w \sigma^2$, and so, by  Cauchy--Schwarz, 
\begin{equation}
    \label{eq:bp-variance} \Exp[ \Delta^2 \mid W_0 = w ] = w \sigma^2 + O (w^{1/2} ) .\end{equation}
Also, note that a consequence of the Macinkiewicz--Zygmund inequality 
(see Corollary~8.2 of~\cite[p.~151]{gut}) is that, for a constant $C_p < \infty$,
\[ \Exp \biggl[ \Bigl| \sum_{i=1}^w ( \xi_{n,i} -1) \Bigr|^p \biggr] \leq C_p w^{p/2} , \text{ for all } w \in \ZP.\]
Hence by Minkowski's inequality, there is a constant $C'_p < \infty$ such that
\begin{equation}
\label{eq:m-z-bound}
\Exp  [ | \Delta |^p \mid W_0 = w ]  \leq C'_p (1 +w)^{p/2} , \text{ for all } w \in \ZP. \end{equation}
Moreover, for $k \in \{1,2\}$ and $\delta \in (0,1)$,
\begin{align}
\label{eq:bp-big-jump-1}
    \Exp [ | \Delta|^k \1 { | \Delta | > w^{1-\delta} } \mid W_0 = w ]
& \leq w^{(1-\delta)(k-p)} \Exp [ | \Delta|^p   \mid W_0 = w ] \nonumber\\
& = O ( w^{(p/2)+(1-\delta)(k-p)} ) ,
\end{align}
by~\eqref{eq:m-z-bound}. Here
\[ \frac{p}{2} + (1-\delta)(k-p) - \left( k - 1 \right) = 1 - \frac{p}{2} + \delta (p- k) , \]
which is strictly negative provided $\delta < \frac{p-2}{2(p-k)}$.
In particular, if we fix $\delta \in (0, \frac{p-2}{2p-2} )$,
then the $k \in \{1,2\}$ cases of~\eqref{eq:bp-big-jump-1} combined
with~\eqref{eq:bp-drift} and~\eqref{eq:bp-variance} show that 
there exists $\eps>0$, depending only on $p$ and $\delta$, for which 
\begin{align}
\label{eq:bp-drift-1}
\Exp [ \Delta \1 { | \Delta | \leq w^{1-\delta} }   \mid W_0 = w ] = \theta + O (w^{-\eps} ) ; \\
   \label{eq:bp-variance-1}
\Exp [ \Delta^2 \1 { | \Delta | \leq w^{1-\delta} }   \mid W_0 = w ] = w \sigma^2 + O (w^{1-\eps} ) .
\end{align}

Some elementary calculus shows that
\begin{equation}
    \label{eq:sqrt-bound} 
\bigl| \sqrt{\max\{1+y,0\}} -1 \bigr| \leq | y |, \text{ for all } y \in \R .\end{equation}
Then, using~\eqref{eq:sqrt-bound} with $y = \Delta$, we obtain
\begin{align*}
\Exp [  | X_{1} - X_0 |^p   \mid W_0 = w ] 
& =  \Exp [  | \sqrt{\max\{w+\Delta,0\}} -\sqrt{w}  |^p   \mid W_0 = w ] \\
& =  w^{p/2} \Exp [  | \sqrt{\max\{ 1 + (\Delta /w) ,0\} } -1  |^p   \mid W_0 = w ] \\
& \leq w^{-p/2} \Exp [  | \Delta  |^p   \mid W_0 = w ],
\end{align*}
which is uniformly bounded in $w \geq 1$, by~\eqref{eq:m-z-bound}. This verifies~\eqref{eq:bp-moments}.

Let $\delta \in (0, \frac{p-2}{2p-2})$. Then, using~\eqref{eq:sqrt-bound} once more,
for $k \in \{1,2\}$,
\begin{align}
\label{eq:bp-big-jump} 
\Exp [  | X_{1} - X_0 |^k \1 { | \Delta | > w^{1-\delta} }  \mid W_0 = w ] 
& \leq w^{-k/2} \Exp [ | \Delta  |^k   \1 { | \Delta | > w^{1-\delta} }  \mid W_0 = w ] \nonumber\\
& \leq w^{-k/2} w^{(1-\delta)(k-p)} \Exp [ | \Delta  |^p  \mid W_0 = w ] \nonumber\\
& =O ( w^{(p-k)(2\delta -1)/2} ),
\end{align}
by~\eqref{eq:m-z-bound}. 
Similarly to above, the constraint $\delta < \frac{p-2}{2(p-k)}$ and 
the $k \in \{1,2\}$ cases of~\eqref{eq:bp-big-jump} show that 
there exists $\eps>0$ for which 
\begin{align}
\label{eq:mu1-big-jump}
   \Exp [  | X_{1} - X_0 | \1 { | \Delta | > w^{1-\delta} }  \mid W_0 = w ] 
& = O ( w^{-(1/2)-\eps} ) ; \\
\label{eq:mu2-big-jump}
\Exp [  | X_{1} - X_0 |^2 \1 { | \Delta | > w^{1-\delta} }  \mid W_0 = w ] 
& = O ( w^{-\eps} ).
\end{align}
On the other hand, on ${ | \Delta | \leq w^{1-\delta} } $, we have from Taylor's theorem that
\begin{align}
\label{eq:bp-taylor}
& ( \sqrt{w + \Delta} - \sqrt{w} ) \1 { | \Delta | \leq w^{1-\delta} } 
 = \sqrt{w} \left( \left( 1 + \frac{\Delta}{w} \right)^{1/2} - 1 \right) \1 { | \Delta | \leq w^{1-\delta} } \nonumber\\
& \qquad \qquad {} =  
\sqrt{w} \left[ \frac{\Delta}{2 w}  
- \frac{\Delta^2}{8 w^2} \left( 1 + O  ( w^{-\delta} ) \right) \right] \1 { | \Delta | \leq w^{1-\delta} }. 
\end{align}
Using the estimates~\eqref{eq:bp-drift} and~\eqref{eq:bp-variance}, we  take expectations in~\eqref{eq:bp-taylor} to
obtain
\begin{equation} 
\label{eq:mu1-small-jump}
\Exp [  ( X_{1} - X_0 ) \1 { | \Delta | \leq w^{1-\delta} }  \mid W_0 = w ]
= \frac{\theta}{2 \sqrt{w}} -\frac{\sigma^2}{8\sqrt{w}} + O (w^{-(1/2)-\eps}),
\end{equation}
for some $\eps>0$ (depending on $\delta$). 
Combining~\eqref{eq:mu1-small-jump} and~\eqref{eq:mu1-big-jump} 
establishes the claimed estimate for $\mu_1$. The  estimate for $\mu_2$ is proved in a similar way,
starting from squaring both sides of~\eqref{eq:bp-taylor}. \end{proof}

\appendix
\section{Characterizing strong transience}
\label{sec:appendix}

This appendix presents the proof of Lemma~\ref{lem:equivalence};
we work under the assumption~\eqref{ass:markov}. 
For $x \in S$ set $N (x) := \Exp_x \sum_{n \in \ZP} \1 { X_n = x }$.
First we give a preliminary result.

\begin{lemma}
Suppose that~\eqref{ass:markov} holds, and $\beta >0$.
For any $x, y \in S$,
\begin{align}
\label{eq:T-ineq}
 T_\beta (x, y ) \Pr_y ( \tau_x < \tau_y ) & \leq T_\beta (y) ; \\
\label{eq:last-exit}
 T_\beta (x, y) & \leq  L_\beta (x, y) \leq U_\beta (x,y) .
\end{align}
Moreover, there exist constants $a_\beta, A_\beta$ with $0 < a_\beta \leq A_\beta < \infty$ such that
\begin{align}
\label{eq:U-T-lower} 
U_\beta (x,y ) & \geq a_\beta \left[ N (y) T_\beta(x,y) + (1 + U_\beta (y) ) \Pr_x ( \tau_y < \infty ) \right] ;\\
\label{eq:U-T-upper} 
U_\beta (x,y ) & \leq A_\beta \left[ N (y) T_\beta(x,y) + (1 + U_\beta (y) ) \Pr_x ( \tau_y < \infty ) \right].\end{align}
\end{lemma}
\begin{proof}
Let $\beta >0$. Note that for $x, y \in S$,
\[ T_\beta (y) = \Exp_y \left[ \tau^\beta_y \1 { \tau_y < \infty } \right] \geq \Exp_y \left[ ( \tau_y -\tau_x )^\beta \1{ \tau_x < \tau_y < \infty } \right] .\]
Hence, by the strong Markov property applied at time $\tau_x$, for $y \neq x$, 
\begin{align*}
T_\beta (y)    \geq \Exp_y \left[ \1 { \tau_x < \tau_y } \Exp_x [ \tau^\beta_y \1 { \tau_y < \infty } ] \right]    =  T_\beta (x, y ) \Pr_y ( \tau_x < \tau_y ) , \end{align*}
as claimed in~\eqref{eq:T-ineq}. 
Since $X_{\lambda_y} = y$ provided $\lambda_y > 0$, we have
\[ \tau_y^\beta \1 { \tau_y < \infty} \leq \lambda_y^\beta \leq \sum_{n \in \N} n^\beta \1 { X_n = y} , \text{ for all } y \in S, \]
which, on taking expectations with respect to $\Pr_x$, yields~\eqref{eq:last-exit}.

Jensen's inequality shows that, for all $a, b \in \RP$,
\begin{equation}
\label{eq:Jensen}
 2^{-(1-\beta)^+} =: a_\beta  \leq \frac{(a+b)^\beta}{a^\beta + b^\beta} \leq A_\beta := 2^{(\beta-1)^+} .\end{equation}
Thus $n^\beta \geq a_\beta ( \tau^\beta_y + (n-\tau_y)^\beta )$ if $\tau_y \leq n$, so that
\begin{align*}
\sum_{n \geq 1} n^\beta  \1 { X_n = y } & \geq a_\beta  \1 { \tau_y < \infty}  \left[
\tau^\beta_y \sum_{n \geq \tau_y} \1 { X_n = y} +  \sum_{n \geq \tau_y }  ( n - \tau_y)^\beta  \1 { X_n = y}\right]  .\end{align*}
By the strong Markov property, on $\{ \tau_y < \infty\}$, $\tau_y$ and $(X_n; n \geq \tau_y )$ are independent and $(X_n; n \geq \tau_y )$ has the same distribution as $(X_n; n \geq 0)$ started from $X_0 =y$, so
\begin{align*}
& {} U_\beta (x,y) \geq \\
& {} \quad {} a_\beta \Exp_x [ \tau^\beta_y \1 { \tau_y < \infty} ] \Exp_y   \sum_{n \geq 0} \1 { X_n = y}
+  a_\beta \Pr_x ( \tau_y < \infty ) \left[ 1 + \Exp_y  \sum_{n \geq 1 } n^\beta \1 { X_n = y} \right] .\end{align*}
This gives the lower bound in~\eqref{eq:U-T-lower},
recalling that $U_\beta (y)$ does not count the visit to~$y$ at time~$0$. The upper bound in~\eqref{eq:U-T-upper}
follows by a similar argument using the upper bound in~\eqref{eq:Jensen}. 
\end{proof}

\begin{proof}[Proof of Lemma~\ref{lem:equivalence}.]
Suppose that $U_\beta (y) < \infty$ for some $y \in S$. Since $N(y) \geq 1$,
it follows from the $y=x$ case of~\eqref{eq:U-T-lower} that $T_\beta (y) \leq a_\beta^{-1} U_\beta (y) < \infty$.
Hence~\ref{lem:equivalence-v} $\Rightarrow$~\ref{lem:equivalence-i}. 
	If $X$ is irreducible, then $\Pr_y ( \tau_x < \tau_y ) > 0$ (or else it would be impossible to reach $x$ from $y$),
	and thus~\eqref{eq:T-ineq} shows that $T_\beta (y) < \infty$ implies that $T_\beta (x,y) < \infty$ for any $x \in S$, so~\ref{lem:equivalence-i} $\Rightarrow$~\ref{lem:equivalence-ii}.
	Assuming that $X$ is transient, $N(y) < \infty$ for all $y \in S$, and so~\eqref{eq:U-T-upper} implies that 
$U_\beta (x,y) < \infty$ whenever~\ref{lem:equivalence-v} holds. 
So we have shown that~\ref{lem:equivalence-v} $\Rightarrow$~\ref{lem:equivalence-i} $\Rightarrow$~\ref{lem:equivalence-ii}  $\Rightarrow$~\ref{lem:equivalence-vi}. Clearly~\ref{lem:equivalence-vi} $\Rightarrow$~\ref{lem:equivalence-v}. This establishes that \ref{lem:equivalence-v} $\Leftrightarrow$~\ref{lem:equivalence-i} $\Leftrightarrow$~\ref{lem:equivalence-ii}  $\Leftrightarrow$~\ref{lem:equivalence-vi}.
Finally~\eqref{eq:last-exit} shows that~\ref{lem:equivalence-vi} $\Rightarrow$~\ref{lem:equivalence-iv} $\Rightarrow$~\ref{lem:equivalence-ii}
and~\ref{lem:equivalence-v} $\Rightarrow$~\ref{lem:equivalence-iii} $\Rightarrow$~\ref{lem:equivalence-i}. This
collection of implications completes the proof.
\end{proof}

\section*{Acknowledgements}
\addcontentsline{toc}{section}{Acknowledgements}

The authors are grateful to two anonymous referees whose attention and suggestions 
led to numerous improvements and clarifications, including  drawing our attention to the related work of~\cite{kz}
and hence 
prompting us to address the link to branching process with migration. 
MM and AW were supported by EPSRC grant EP/W00657X/1.


\begin{thebibliography}{00}

\bibitem{alexander}
K.\ Alexander, 
Excursions and local limit theorems for Bessel-like random walks.
\emph{Electron.\ J.\ Probab.}\ {\bf 16} (2011) 1--44.

\bibitem{aim}
S.\ Aspandiiarov, R.\ Iasnogorodski, and M.\ Menshikov, 
Passage-time moments for nonnegative stochastic processes and
an application to reflected random walks in a quadrant.
\emph{Ann.\ Probab.}\ {\bf 24} (1996) 932--960.

\bibitem{as}
A.\ Asselah and B.\ Schapira,
The two regimes of moderate deviations for the range of a transient random walk.
\emph{Probab.\ Theory Related Fields} {\bf 180} (2021) 439--465.

\bibitem{bd}
J.\ Bertoin and R.A.\ Doney,
Some asymptotic results for transient random walks.
\emph{Adv.\ in Appl.\ Probab.}\ {\bf 28} (1996) 207--226.

\bibitem{bdindo}
G.\ Bottazzi and P.\ Dindo,
Drift criteria for persistence of discrete stochastic processes on the line.
\emph{J.\ Math.\ Economics} {\bf 101} (2022) 102696.


\bibitem{BR}
R.N.\ Bhattacharya and R.R.\ Rao, 
\emph{Normal Approximation and Asymptotic Expansions}.
Updated reprint of the 1986 edition,
 Classics in Applied Mathematics, 
SIAM, Philadelphia, 2010.

\bibitem{bgt}
N.H.\ Bingham, C.M.\ Goldie, and J.L.\ Teugels.
\emph{Regular Variation}. Cambridge University Press, Cambridge, 1989.

\bibitem{dkw1}
D.\ Denisov, D.\ Korshunov, and V.\ Wachtel,
Potential analysis for positive recurrent Markov chains with asymptotically zero drift: Power-type asymptotics.
\emph{Stochastic Process.\ Appl.}\ {\bf 123} (2013) 3027--3051.

\bibitem{dkw-book}
D.\ Denisov, D.\ Korshunov, and V.\ Wachtel,
\emph{At the Edge of Criticality: Markov Chains with Asymptotically Zero Drift}.
arXiv:1612.01592 (2016).

\bibitem{dkw2}
D.\ Denisov, D.\ Korshunov, and V.\ Wachtel,
Renewal theory for transient Markov chains with asymptotically zero drift.
\emph{Trans.\ Amer.\ Math.\ Soc.}\ {\bf 373} (2020) 7253--7286.

\bibitem{dk}
R.A.\ Doney and D.A.\ Korshunov, 
Local asymptotics for the time of first return to the origin of transient random walk. 
\emph{Stat.\ Probab.\ Lett.}\ {\bf 81} (2011) 1419--1424.

\bibitem{de} A.\ Dvoretzky and P.\ Erd\H os, 
Some problems on random walk in space.
pp.~353--367 in: 
Proc.\ Second Berkeley Symp.\ on Math.\ Statist.\ and Probab., Vol.~2, Univ.\ of Calif.\ Press, 1951.

\bibitem{et} 
P.\ Erd\H os and S.J.\ Taylor, Some intersection properties of random walk paths.
\emph{Acta Math.\ Sci.\ Hung.}\ {\bf 11} (1960) 231--248.

\bibitem{esseen}
C.G.\ Esseen,
On the concentration function of a sum of independent random variables.
\emph{Z.\ Wahrscheinlichkeitstheorie verw.\ Geb.}\ {\bf 9} (1968) 290--308.

\bibitem{fal}
A.M.\ Fal$'$, 
Certain limit theorems for an elementary Markov random walk.
\emph{Ukrainian Math.\ J.}\ {\bf 33} (1981) 433--435.
Translated from \emph{Ukrainskii Mat.\ Z.}\ {\bf 33} (1981) 564--566.

\bibitem{gallardo}
L.\ Gallardo,
Comportement asymptotique des marches aleatoires associees aux polynomes de
Gegenbauer et applications.
\emph{Adv.\ in Appl.\ Probab.}\ {\bf 16} (1984) 293--323.

\bibitem{gw}
N.\ Georgiou and A.R.\ Wade,
Non-homogeneous random walks on a semi-infinite strip.
\emph{Stoch.\ Process.\ Appl.}\ {\bf 124} (2014) 3179--3205. 

\bibitem{gut} 
A.\ Gut,
\emph{Probability: A Graduate Course.}
\emph{Springer}, Berlin, 2005.

\bibitem{hughes}
B.D.\ Hughes,
\emph{Random Walks and Random Environments. Volume 1: Random Walks.}
Clarendon Press, Oxford, 1995.

\bibitem{jo}
N.C.\ Jain and S.\ Orey, 
On the range of random walk.
\emph{Israel J.\ Math.}\ {\bf 6} (1968) 373--380.

\bibitem{jp72a}
N.C.\ Jain and W.E.\ Pruitt,
The range of random walk.
pp.~31--50 in: Proc.\ Sixth Berkeley Symp.\ on Math.\ Statist.\ and Probab., Vol. 3, Univ.\ of Calif.\ Press, 1972.

\bibitem{joffe}
A.\ Joffe,
\emph{Sojourn Time for Stable Processes}. Thesis, Cornell University, 1959.

\bibitem{kz}
E.\ Kosygina and M.P.W.\ Zerner,
Excursions of excited random walks on integers.
\emph{Electron.\ J.\ Probab.}\ {\bf 19} (2014), no.~25.

\bibitem{lamp1} J.\ Lamperti, Criteria for the recurrence and transience of stochastic processes I,
\emph{J.\ Math.\ Anal.\ Appl.}\ {\bf 1} (1960) 314--330.

\bibitem{lamp2} J.\ Lamperti, 
A new class of probability limit theorems.
\emph{J.\ Math.\ Mech.}\ {\bf 11} (1962) 749--772.

\bibitem{lamp3}
J.\ Lamperti, 
Criteria for
stochastic processes II: passage-time moments.
\emph{J.\ Math.\ Anal.\ Appl.}\ {\bf 7} (1963) 127--145.

%\bibitem{ll}
%G.F.\ Lawler and V.\ Limic,
%\emph{Random Walk: A Modern Introduction.}
%Cambridge University Press, Cambridge, 2010.

\bibitem{lmw}
C.H.\ Lo, M.V.\ Menshikov, and A.R.\ Wade, 
Cutpoints of non-homogeneous random walks. 
\emph{ALEA Latin Amer.\ J.\ Probab.\ Math.\ Statist.}\ {\bf 19} (2022) 493--510.

\bibitem{lw}
C.H.\ Lo and A.R.\ Wade, 
On the centre of mass of a random walk.
\emph{Stoch.\ Process.\ Appl.}\  {\bf 129} (2019) 4663--4686.

\bibitem{mai} M.V. Menshikov, I.M. Asymont, and R. Iasnogorodskii,
Markov processes with asymptotically zero drifts.
{\em Problems of
Information Transmission} {\bf 31} (1995) 248--261; translated from
{\em Problemy Peredachi Informatsii} {\bf 31} (1995) 60--75 (Russian).

\bibitem{mpw} M.\ Menshikov, S.\ Popov, and A.\ Wade, 
\emph{Non-homogeneous Random Walks}.
Cambridge University Press, 
Cambridge, 2017.	

\bibitem{pakes}
A.G.\ Pakes,
On the critical Galton-Watson process with immigration.
\emph{J.\ Austral.\ Math.\ Soc.}\ {\bf 12} (1971) 476--482.

\bibitem{port66}
S.C.\ Port,
Limit theorems involving capacities.
\emph{J.\ Math.\ Mech.}\ {\bf 15} (1966) 805--832.

\bibitem{port67}
S.C.\ Port,
Limit theorems for transient Markov chains.
\emph{J.\ Combinat.\ Theory} {\bf 2} (1967) 107--128.

\bibitem{rosenkrantz}
W.A.\ Rosenkrantz,
A local limit theorem for a certain class of random walks.
\emph{Ann.\ Math.\ Statist.}\ {\bf 37} (1966) 855--859.

\bibitem{sandric}
N.\ Sandri\'c,
On transience of L\'evy-type processes.
\emph{Stochastics} {\bf 88} (2016) 1012--1040.

\bibitem{sato}
K.-I.\ Sato,
Criteria of weak and strong transience for L\'evy processes.
pp.~438--449 in: S.~Watanabe \emph{et al.}~(eds.) \emph{Probability Theory and Mathematical Statistics}, Proc.\ Seventh Japan--Russia Symp., World Scientific, Singapore, 1996.

\bibitem{sw}
K.-I.\ Sato and T.\ Watanabe,
Moments of last exit times for L\'evy processes.
\emph{Ann.\ Inst.\ H.\ Poincar\'e Probab.\ Statist.}\ {\bf 40} (2004) 207--225.

\bibitem{sw2}
K.-I.\ Sato and T.\ Watanabe,
Last exit times for transient semistable processes.
\emph{Ann.\ Inst.\ H.\ Poincar\'e Probab.\ Statist.}\ {\bf 41} (2005) 929--951.

\bibitem{spitzer64}
F.\ Spitzer,
Electrostatic capacity, heat flow, and Brownian motion.
\emph{Z.\ Wahrscheinlichkeitstheorie verw.\ Geb.}\ {\bf 3} (1964) 110--121.

\bibitem{spitzerbook}
F.\ Spitzer, \emph{Principles of Random Walk}. 2nd ed., Springer, New York, 1976.

\bibitem{takeuchi}
J.\ Takeuchi,
Moments of the last exit times.
\emph{Proc.\ Japan Acad.}\ {\bf 43} (1967) 355--360.

\end{thebibliography}
\end{document}